\documentclass[final]{siamart171218}
\usepackage{braket,amsfonts}
\usepackage{array}

\allowdisplaybreaks[1]

\usepackage[caption=false]{subfig}
\newsiamthm{claim}{Claim}
\newsiamremark{remark}{Remark}
\newsiamremark{hypothesis}{Hypothesis}
\crefname{hypothesis}{Hypothesis}{Hypotheses}

\Crefname{ALC@unique}{Step}{Step}

\usepackage{algorithmic}
\renewcommand{\algorithmiccomment}[1]{\hfill $\triangleright$ \emph{#1}}

\usepackage{graphicx,epstopdf}

\usepackage{amsopn}

\usepackage{enumitem}
\newcounter{example}
\newenvironment{example}{\refstepcounter{example}\vspace{1ex}
{\sc Example \theexample.}\hspace{0.3em}\parindent=0pt}{\vspace{1ex}}

\begin{document}

\title{Analytical low-rank compression via proxy point selection\thanks{Submitted for review.
\funding{The research of Jianlin Xia was supported in part by an NSF grant DMS-1819166.}}}

\author{Xin Ye\thanks{Department of Mathematics, Purdue University, West Lafayette, IN 47907 (\email{ye83@purdue.\allowbreak edu}, \email{xiaj@purdue.\allowbreak edu}).}
\and Jianlin Xia\footnotemark[2]
\and Lexing Ying\thanks{Department of Mathematics and Institute for Computational and Mathematical Engineering, Stanford University, Stanford, CA 94305 (\email{lexing@stanford.edu}).}}

\headers{X. Ye, J. XIA, and L. Ying}{Low-rank compression via proxy point selection}

\maketitle

\begin{abstract}
It has been known in potential theory that, for some kernels matrices corresponding to well-separated point sets, fast analytical low-rank
approximation can be achieved via the use of proxy points. This proxy point method gives a surprisingly convenient
way of explicitly writing out approximate basis matrices for a kernel matrix. However, this elegant strategy is rarely
known or used in the numerical linear algebra community. It still needs clear algebraic understanding of the theoretical background.
Moreover, rigorous quantifications of the approximation errors and reliable criteria for the selection of the proxy points
are still missing.
In this work, we use contour integration to clearly justify the idea in terms of a class of important kernels.
We further provide comprehensive accuracy analysis for the analytical compression and show how to choose nearly optimal proxy points.
The analytical compression is then combined with fast rank-revealing factorizations to get compact low-rank approximations
and also to select certain representative points. We provide the error bounds for the resulting overall low-rank approximation.
This work thus gives a fast and reliable strategy for compressing those kernel matrices. Furthermore,
it provides an intuitive way of understanding the proxy point method
and bridges the gap between this useful analytical strategy and practical low-rank approximations.
Some numerical examples help to further illustrate the ideas.
\end{abstract}



\section{Introduction}
\label{sec:intro}

In this paper, we focus on the low-rank approximation of some kernel matrices: those generated by a smooth kernel function $\kappa(x,y)$ evaluated at two well-separated sets of points $X = \{ x_j \}_{j=1}^m $ and $Y = \{ y_j \}_{j=1}^n$.
We suppose $\kappa(x,y)$ is analytic and a degenerate approximation as follows exists:
\begin{equation}\label{eqn:degenerate}
\kappa(x,y) \approx  \sum_{j=1}^{r} \alpha_j \psi_j(  x) \varphi_j(y),
\end{equation}
where $\psi_j$'s and $\varphi_j$'s are appropriate basis functions and $\alpha_j$'s are coefficients independent of $x$ and $y$. $X$ and $Y$ are well separated in the sense that the distance between them is comparable to their diameters so that $r$ in (\ref{eqn:degenerate}) is small.
In this case, the corresponding discretized kernel matrix as follows is numerically low rank:
\begin{equation}\label{eq:k}
K^{(X,Y)}\equiv(\kappa(x,y)_{x\in X,y\in Y}).
\end{equation}

This type of problems frequently arises in a wide range of computations such as numerical solutions of PDEs and integral equations, Gaussian processes, regression with massive data, machine learning, and $N$-body problems. The low-rank approximation to $K^{(X,Y)}$ enables
fast matrix-vector multiplications in methods such as the fast multipole method (FMM) \cite{greengard1987fast}. It can also be used
to quickly compute matrix factorization and inversion based on rank structures such as $\mathcal{H}$ \cite{hackbusch1999sparse}, $\mathcal{H}^2$ \cite{bor02,hackbusch2000h2}, and HSS \cite{chandrasekaran2006fast,xia2010fast} forms.
In fact, relevant low-rank approximations play a key role in rank-structured methods. The success of the so-called fast rank-structured direct solvers relies heavily on the quality and efficiency of low-rank approximations.

According to the Eckhart-Young Theorem \cite{eckart1936approximation}, the best 2-norm low-rank approximation is given by the truncated SVD,
which is usually expensive to compute directly. More practical \emph{algebraic compression} methods include rank-revealing factorizations
(especially strong rank-revealing QR \cite{gu1996efficient} and strong rank-revealing LU factorizations
\cite{miranian2003strong}), mosaic-skeleton
approximations \cite{tyr96}, interpolative decomposition \cite{cheng2005compression}, CUR decompositions \cite{mah09}, etc.
Some of these algebraic methods have a useful feature of {\it structure preservation} for $K^{(X,Y)}$:
relevant resulting basis matrices can be submatrices of the original
matrix and are still discretizations of $\kappa(x,y)$ at some subsets. This is a very useful feature that can
greatly accelerate some hierarchical rank structured direct solvers \cite{toeprs,liu2016parallel,mfhssrs}.
However, these algebraic compression methods have
$\mathcal{O}(rmn)$ complexity and are very costly for large-scale applications. The efficiency may be improved by randomized SVDs \cite{halko2011finding,gu2015subspace,martinsson2017householder}, which still cost $\mathcal{O}(rmn)$ flops.

%A remedy to the algebraic methods is the introduction of randomization. The most general such approach \cite{frieze2004fast} is done by randomized sampling on the rows or columns of the matrix with certain weight, the selected submatrix can capture the majority of the actions in the original matrix with very low failure probability; when a fast matrix-vector product scheme is available, column or row spaces information can be also extracted from the product in a matrix-free fashion \cite{halko2011finding,gu2015subspace,xiao2017fast,martinsson2017householder}. With the assumption that each entries of the matrix can be accessed in $\mathcal{O} (1)$ cost or the mat-vec can be conducted in less than $\mathcal{O} (mn) $ cost, randomized methods generally bring down the compression cost to at most $\mathcal{O} ( (m+n)  r) $ with some possible polylogarithmic term in $m$ or $n$.

Unlike fully algebraic compression, there are also various \emph{analytical compression} methods that take advantage of degenerate approximations
like in (\ref{eqn:degenerate}) to compute low-rank approximations. The degenerate approximations may be obtained by Taylor expansions, multipole expansions \cite{greengard1987fast}, spherical harmonic basis functions \cite{sun2001matrix}, Fourier transforms with Poisson's formula \cite{anderson1992implementation,makino1999another}, Laplace transforms with the Cauchy integral formula \cite{letourneau2014cauchy}, Chebyshev interpolations \cite{fong2009black}, etc. Various other polynomial basis functions may also be used \cite{oneil2007new}.

These analytical approaches can quickly yield low-rank approximations to $K^{(X,Y)}$ by explicitly producing
approximate basis matrices. On the other hand, the resulting low-rank approximations are
usually not structure preserving in the sense that the basis matrices are not directly related to $K^{(X,Y)}$. This is because the basis
functions $\{  \psi_j  \}$ and $\{ \varphi_j \}$ are generally different from $\kappa(x,y)$.

As a particular analytical compression method, the \emph{proxy point method} has attracted a lot of interests in recent years. It is tailored for kernel matrices and is very attractive for different geometries of points \cite{fong2009black,martinsson2005fast,xing2018efficient,ying2006kernel,ying2004kernel}. While the methods vary from one to another, they all share the same basic idea and can be summarized in the surprisingly simple \cref{alg:proxy}, where the details are omitted and will be discussed later in later sections.
Note that an explicit degenerate form \cref{eqn:degenerate} is not needed and the algorithm directly produces the matrix $K^{(X,Z)}\equiv(\kappa(x,y)_{x\in X,y\in Z})$ as an
approximate column basis matrix in \cref{step:compress}. This feature enables the extension of the ideas of the classical fast multipole method (FMM) \cite{greengard1987fast} to more general situations, and examples include the recursive skeletonization \cite{ho2012fast,martinsson2005fast,minden2017recursive} and kernel independent FMM \cite{martinsson2007accelerated,ying2006kernel,ying2004kernel}.

\begin{algorithm}[!h]
    \caption{\it Basic proxy point method for low-rank approximation}
    \emph{Input:} $\kappa(x,y)$, $X$, $Y$ \\
    \emph{Output:} Low-rank approximation $K^{(X,Y)} \approx AB$\mbox{}\hfill\algorithmiccomment{\it Details in \cref{sec:compression,sec:optimal}}
    \label{alg:proxy}
    \vspace{-4mm}
    \begin{algorithmic}[1]
        \STATE\label{step:proxy}{Pick a \emph{proxy surface} $\Gamma$ and a set of \emph{proxy points} $Z \subset \Gamma $}
        \STATE{$A\leftarrow K^{(X,Z)}$}\label{step:compress}
        \STATE{$B\leftarrow \Phi^{(Z,Y)}$ for a matrix $\Phi^{((Z,Y)}$ such that $K^{X,Y)} \approx K^{(X,Z)}  \Phi^{(Z,Y)}$}
    \end{algorithmic}
\end{algorithm}

Notice that $|Z|$ is generally much smaller than $|Y|$ so that $K^{(X,Z)}$ has a much smaller column size than $K^{(X,Y)}$.
It is then practical to apply reliable rank-revealing factorizations to $K^{(X,Z)}$ to extract a compact
approximate column basis matrix for $K^{(X,Y)}$. This is a {\it hybrid (analytical/algebraic) compression} scheme,
and the proxy point method helps to significantly reduce the compression cost.

The significance of the proxy point method can also be seen from another viewpoint: the selection
of {\it representative points}. When a strong rank-revealing QR (SRRQR) factorization or interpolative decomposition is applied to
$K^{(X,Y)}$, an approximate row basis matrix can be constructed from selected rows of $K^{(X,Y)}$. Suppose those rows correspond to
the points $\hat{X}\subset X$. Then $\hat{X}$ can be considered as a subset of representative points. The analytical selection of $\hat{X}$ is
not a trivial task. However, with the use of the proxy points $Z$, we can essentially quickly find $\hat{X}$ based on $K^{(X,Z)}$. (See \cref{sec:hybrid} for more details.) That is, the set of proxy points $Z$ can serve as a set of auxiliary points based on which the representative points can be quickly identified. In another word,
when considering the interaction $K^{(X,Y)}$ between $X$ and $Y$, we can use the interaction $K^{(X,Z)}$ between $X$ and
the proxy points $Z$ to extract the contribution $\hat{X}$ from $X$.

Thus, the proxy point method is a very convenient and useful tool for researchers working on kernel matrices.
However, this elegant method is
much less known in the numerical linear algebra community. Indeed, even the compression of some special Cauchy matrices (corresponding
to a simple kernel) takes quite some efforts in matrix computations \cite{mar05,pan15,toeprs}. In a recent literature survey \cite{kis17} that lists
many low-rank approximation methods (including a method for kernel matrices), the proxy point method is not mentioned at all.
One reason that the proxy point method is not widely known by researchers in matrix computation is the lack of
intuitive algebraic understanding of the background.% though potential theory can be used to explain its feasibility.

Moreover, in contrast with the success of the proxy point method in various practical applications, its theoretical justifications
are still lacking in the literature. Potential theory \cite[Chapter 6]{kress2014linear} can be used to explain the choice of proxy surface $\Gamma$ in \cref{step:proxy} of \cref{alg:proxy} when dealing with some PDE kernels (when $\kappa(x,y)$ is
the fundamental solution of a PDE). However, there is no clear justification
of the accuracy of the resulting low-rank approximation.
Specifically, a clear explanation of such a simple procedure in terms of both the approximation error and the proxy point selection desired, especially from the linear algebra point of view.

Thus, we intend to seek a convenient way to {\it understand the proxy point method and its accuracy} based on some kernels.
The following types of errors will be considered (the notation will be made more precise later):
\begin{itemize}
\item The error $\varepsilon$ for the approximation
of kernel functions $\kappa(x,y)$ with the aid of proxy points.
\item The error $\mathcal{E}$ for the
low-rank approximation of kernel matrices $K^{(X,Y)}$ via the proxy point method.
\item The error $\mathcal{R}$ for
practical hybrid low-rank approximations of $K^{(X,Y)}$ based on the proxy point method.
\end{itemize}

Our main objectives are as follows.
\begin{enumerate}
\item Provide an intuitive explanation of the proxy point method using contour integration so as to
make this elegant method more accessible to the numerical linear algebra community.

\item Give systematic analysis of the approximation errors of
the proxy point method as well as the hybrid compression. We show how the
kernel function approximation error $\varepsilon$ and the low-rank compression error $\mathcal{E}$ decay exponentially
with respect to the number of proxy points. We also show how our bounds for the error $\mathcal{E}$ are nearly
independent of the geometries and sizes of $X$ and $Y$ and why a bound for the error $\mathcal{R}$
may be independent of one set (say, $Y$).

\item Use the error analysis to choose a nearly optimal set of proxy points in the low-rank kernel matrix compression.
Our error bounds give a clear guideline to control the errors and to choose the locations of the proxy points
so as to find nearly minimum errors. We also give a practical method to quickly estimate the optimal locations.
\end{enumerate}

We conduct such studies based on kernels of the form
\begin{equation}\label{eqn:kernel}
    \kappa(x,y) = \frac{1}{(x - y)^d},\quad x,y \in \mathbb{C},\quad x \neq y,
\end{equation}
where $d$ is a positive integer. Such kernels and their variants are very useful in PDE and integral equation solutions,
structured ODE solutions \cite{fastodesol}, Cauchy matrix computations \cite{pan15},
Toeplitz matrix direct solutions \cite{toep,mar05,toeprs}, structured divide-and-conquer Hermitian eigenvalue solutions \cite{gu95,hsseig}, etc.
Our derivations and analysis may also be useful for studying
other kernels and higher dimensions. This will be considered in future work.
(Note that the issue of what kernels the proxy point method can apply to is not the focus here.)

We would like to point out that several of our results like the error analyses in \cref{sec:optimal,sec:hybrid} can
be easily extended to more general kernels and/or with other approximation methods, as long as
a relative approximation error for the kernel function approximation is available.
Thus, our studies are useful for more general situations.

Our theoretical studies
are also accompanied by various intuitive numerical tests which show that the error bounds nicely capture the
error behaviors and also predict the location of the minimum errors.

In the remaining discussions, \cref{sec:compression} is devoted to an intuitive derivation of the
proxy point method via contour integration and the analysis of the accuracy ($\varepsilon$)
for the approximation of the kernel functions. The analytical low-rank compression accuracy ($\mathcal{E}$)
and the nearly optimal proxy point selection are given in \cref{sec:optimal}.
The study is further extended to the analysis of the
hybrid low-rank approximation accuracy ($\mathcal{R}$) with representative point selection in \cref{sec:hybrid}.
Some notation we use frequently in the paper is listed below.
\begin{itemize}
    \item The sets under consideration are $X = \{ x_j \}_{j=1}^m$ and $Y = \{ y_j \}_{j=1}^n$. $Z= \{ z_j \}_{j=1}^N$ is the set of proxy points.
    \item $ \mathcal{C}(a;\gamma)$, $\mathcal{D}(a;\gamma)$, and $ \bar{\mathcal{D}} (a;\gamma)$ denote respectively the circle, open disk, and closed disk with center $a\in \mathbb{C} $ and radius $\gamma > 0 $.

    \item $\mathcal{A} ( a; \gamma_1, \gamma_2)  = \{ z :  \gamma_1 <  |z - a| < \gamma_2   \} $ with $0 < \gamma_1 < \gamma_2$ is an open annulus region.

    \item $K^{(X,Y)}$ is the $m\times n$ kernel matrix $(\kappa(x_i, y_j)_{x_i\in X,y_j\in Y})$ with $\kappa(x,y)$ in \cref{eqn:kernel}. Notation such as $K^{(X,Z)}$ and $K^{(\hat{X},Z)}$ will also be used and can be understood similarly.
\end{itemize}

\section{The proxy point method for kernel function approximation and its accuracy}
\label{sec:compression}

In this section, we show one intuitive derivation of the proxy point method for the analytical approximation of the kernel functions,
followed by detailed approximation error analysis.

Note that the kernel function \cref{eqn:kernel} is translation invariant, i.e., $\kappa(x-z, y - z) = \kappa(x,y)$ for any $x \neq y $
and $z \in \mathbb{C}$. Thus, the points $X$ can be moved to be clustered around the origin. Without loss of generality,
we always assume $X \subset \mathcal{D} (0;\gamma_1)$ and $Y \subset \mathcal{A} ( 0; \gamma_2, \gamma_3) $, where the radii satisfy $0 < \gamma_1 < \gamma_2 < \gamma_3$. See \cref{fig:1}. Such situations arise frequently in applications of the FMM.
\begin{figure}[h]
\centering
\includegraphics[height=1.8in]{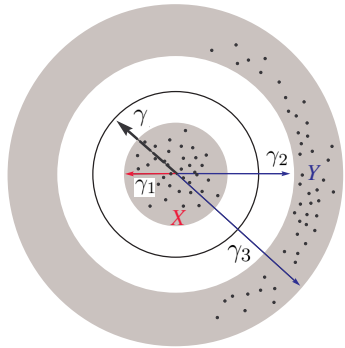}
\caption{Illustration of $\gamma$, $\gamma_1$, $\gamma_2$, $\gamma_3$, $X$, and $Y$.}
\label{fig:1}
\end{figure}

\subsection{Derivation of the proxy point method via contour integration}
\label{subsec:derivation}

Consider any two points $x \in X$ and $y \in Y$. Draw a Jordan curve (a simple closed curve) $\Gamma$ that encloses $x$ while excluding $y$, and let $\rho>0$ be large enough so that the circle $\mathcal{C}(0;\rho)$ encloses both $\Gamma$ and $y$. See \cref{fig:2a}.
\begin{figure}[ptbh]
\centering
\subfloat[$\Gamma$ and $\mathcal{C}(0;\rho)$ used in contour integration]{\label{fig:2a}\qquad\includegraphics[height=1.8in]{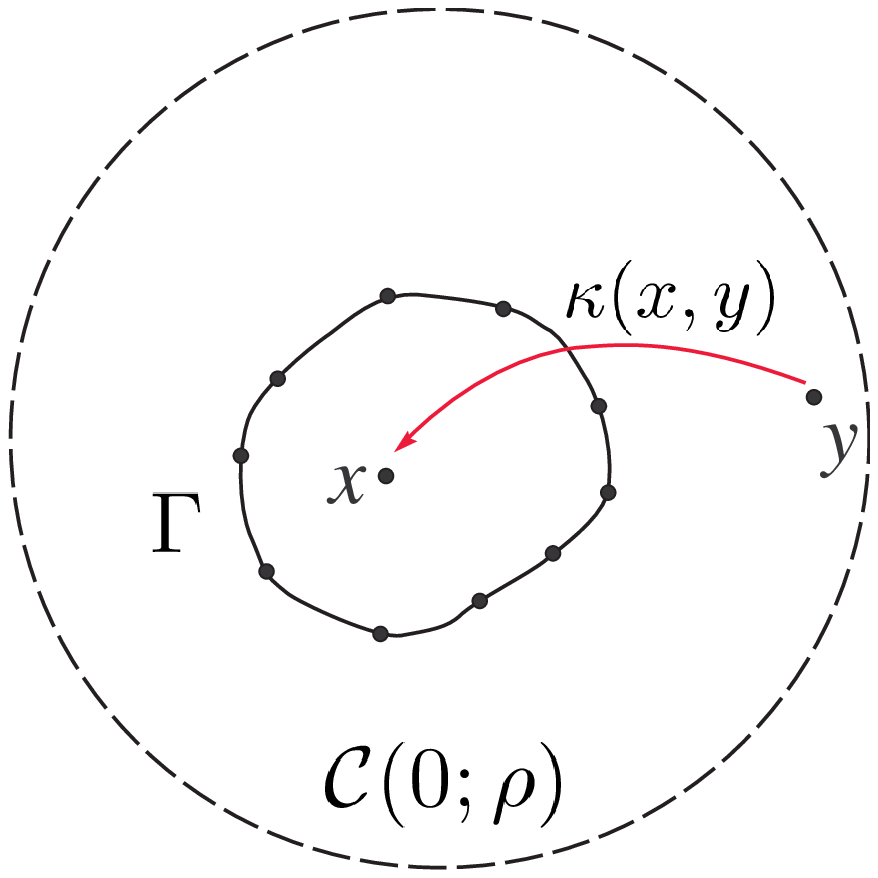}\qquad}
\subfloat[Approximation of $\kappa(x,y)$]{\label{fig:2b}\quad\includegraphics[height=1.8in]{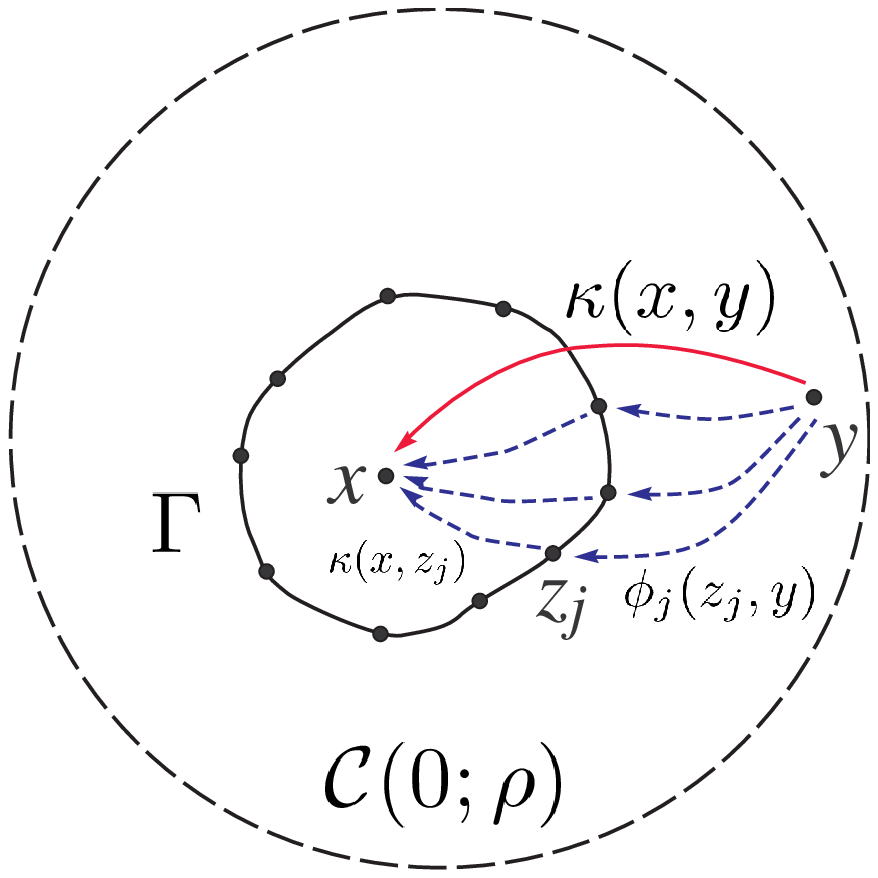}\quad}
\caption{Approximating the interaction $\kappa(x,y)$ by $\tilde{\kappa} (x,y)$ in \cref{eqn:kn} using proxy points.}
\label{fig:2}
\end{figure}

Define the domain $\Omega_{\rho}$ to be the open region inside $\mathcal{C}(0;\rho)$ and outside $\Gamma$. Its boundary is $ \partial \Omega_{\rho}:=\mathcal{C}(0;\rho)\cup(-\Gamma)$, where $-\Gamma$ denotes the curve $\Gamma$ in its negative direction. Now consider the function $f(z) := \kappa(x,z)$ on the closed domain $\bar{\Omega}_{\rho}  := \Omega_{\rho} \cup \partial \Omega_{\rho}$. The only singularity of $f(z)$ is at $z = x \notin \bar{\Omega}_{\rho}$. Thus, $f(z)$ is analytic (or holomorphic) on $\bar{\Omega}_{\rho}$. By the Cauchy integral formula \cite{stein2003complex},
\begin{equation}\label{eqn:cauchy}
\kappa(x,y)  = f(y)  = \frac{1}{ 2 \pi \mathbf{i}} \int_{\partial \Omega_{\rho}}   \frac{f(z)}{z - y} \mathrm{d} z\\
  = \frac{1}{ 2 \pi \mathbf{i}} \int_{\mathcal{C}(0;\rho)} \frac{\kappa(x,z)}{z - y} \mathrm{d} z  - \frac{1}{ 2 \pi \mathbf{i}} \int_{\Gamma}   \frac{\kappa(x,z)}{z-y} \mathrm{d} z,
\end{equation}
where $\mathbf{i}=\sqrt{-1}$.
Note that
\begin{align*}
\left| \int_{\mathcal{C}(0;\rho)}   \frac{\kappa(x,z)}{z - y} \mathrm{d} z  \right| &   \leq 2 \pi {\rho} \cdot \max_{z \in {\mathcal{C}(0;\rho)}} \left| \frac{1}{(x-z)^d(z - y)} \right| \leq \frac{2 \pi {\rho}}{ (\rho - |x|)^d(\rho - |y|) },
\end{align*}
where the right-hand side goes to zero when $\rho\to \infty $. Thus,
\begin{equation*}
\lim_{{\rho} \to \infty } \int_{\mathcal{C}(0;\rho)}   \frac{\kappa(x,z)}{z - y} \mathrm{d} z  = 0.
\end{equation*}
Take the limit on \cref{eqn:cauchy} for ${\rho} \to \infty$, and the first term on the right-hand side vanishes. We get
\begin{equation}\label{eqn:coutour_integral}
\kappa(x,y) = \frac{1}{ 2 \pi \mathbf{i}} \int_\Gamma   \frac{\kappa(x,z)}{y - z} \mathrm{d} z.
\end{equation}
Note that this result is different from the Cauchy integral formula in that the point $y$ under consideration is outside the contour $\Gamma$ in the integral.

To numerically approximate the contour integral \cref{eqn:coutour_integral}, pick an $N$-point quadrature rule with quadrature points $\{z_j \}_{j=1}^N \subset \Gamma$ and
the corresponding quadrature weights $ \{ \omega_j  \}_{j=1}^N$. Denoted by $\tilde{\kappa}( x,y)$ the approximation induced by
such a quadrature integration:
\begin{equation}\label{eqn:kn}
\tilde{\kappa} (x,y) = \frac{1}{2 \pi \mathbf{i}} \sum_{j=1}^{ N} \omega_j \frac{ \kappa(x,z_j) }{y - z_j}  \equiv \sum_{j=1}^{N}  \kappa(x, z_j) \phi_j(z_j,y),\quad\text{with}\quad \phi_j(z,y) = \frac{\omega_j}{2 \pi \mathbf{i} (y - z)}.
\end{equation}

Clearly, $\tilde{\kappa} (x,y)$ in \cref{eqn:kn} is a degenerate approximation  to $\kappa(x,y)$ like \cref{eqn:degenerate}. Moreover, it has one
additional property of {\it structure preservation}: the function $\varphi_j(x)$ in this case is $\kappa(x,z_j)$, which is exactly the
original kernel $\kappa(x,y)$ with $z_j$ in the role of $y$. This gives a simple and intuitive explanation of the use of proxy points: the
interaction between $x$ and $y$ can essentially be approximated by the interaction between $x$ and some proxy points $Z$ (and later we
will further see that $Z$ can be independent of the number of $x$ and $y$ points). These two interactions are made equivalent
(in terms of computing potential) through the use
of the function $\phi_j$. In another word, equivalent charges can be placed on the proxy surface. A pictorial illustration is shown in \cref{fig:2b}.

\subsection{Approximation error analysis}
\label{subsec:error}

Although the approximation \cref{eqn:kn} holds for any proxy surface $\Gamma$ satisfying the given conditions and for any quadrature rule, we still need to make specific choices in order to obtain a more practical error bound. Firstly, we assume the proxy surface to be a circle: $\Gamma = \mathcal{C} (0;\gamma) $, which is on of the most popular choices in related work and is also consistent with our assumptions at the beginning of \cref{sec:compression}. For now, the proxy surface $\Gamma$ is only assumed to be between $X$ and $Y$, i.e., $\gamma_1 < \gamma < \gamma_2$ as in \cref{fig:1}, and we will come back to discuss more on this later. Secondly, the quadrature rule is chosen to be the composite trapezoidal rule with
\begin{equation}\label{eqn:trapezoidal}
z_j =  \gamma \exp\left(\frac{2 j\pi \mathbf{i} }{N} \right),\quad \omega_j = \frac{2 \pi \mathbf{i} }{N}z_j, \quad j = 1,2,\dots,N.
\end{equation}
This choice can be justified by noting that the trapezoidal rule converges exponentially fast if applied to a periodic integrand \cite{trefethen2014exponentially}. Our results later also align with this. Moreover, if no specific direction is more important that others, the trapezoidal rule performs uniformly well on all directions of the complex plane $\mathbb{C}$. Some related discussions of this issue can be found in \cite{kestyn2016feast,ye2017fast}.

As a result of the above assumptions, the function $\phi_j(z,y)$ in \cref{eqn:kn} becomes the following form:
\begin{equation*}
\phi(z,y) = \frac{1}{N} \frac{z}{y-z},\quad y \neq z,
\end{equation*}
where we dropped the subscript $j$ since $j$ does not explicitly appear on the right-hand side. Also, we define
\[g(z) = \frac1{z-1},\quad z\ne1.\]
%as a complex function for $z \in \mathbb{C}$ and $z \neq 1$ or a real function for $z \in  \mathbb{R}$ and $z > 1$ depending on the context.

The following lemma will be used in the analysis of the approximation error for $\kappa(x,y)$.

\begin{lemma}\label{lem:sum}
Let $\{ z_j \}_{j=1}^N$ be the points defined in \cref{eqn:trapezoidal}. Then the following result holds for all $z \in \mathbb{C} \backslash \{ z_j \}_{j=1}^N$:
\begin{equation}\label{eqn:sum}
\sum_{j=1}^{N} \frac{z_j}{z - z_j} =Ng\!\left(\Big(\frac{z}{\gamma}\Big)^N\right).
\end{equation}
\end{lemma}

\begin{proof}
For any integer $p$, we have
\begin{equation}\label{eqn:sum_zj}
\sum_{j=1}^{N}  z_j^p = \begin{cases}
N \gamma^p, \quad &\text{if $p$ is a multiple of $N$}, \\
0, \quad &\text{otherwise}.
\end{cases}
\end{equation}
If $|z| < \gamma $, then $|z/z_j| < 1$ for $j = 1,2,\dots,N$ and
\begin{align*}
\sum_{j=1}^{N} \frac{z_j}{z - z_j} & = -  \sum_{j=1}^{N} \frac{1}{1 - z/z_j} = - \sum_{j=1}^{N}  \sum_{k=0}^{\infty} \left( \frac{z}{z_j} \right)^k
 = -  \sum_{k=0}^{\infty} \Big( z^k \sum_{j=1}^{N} z_j^{-k}\Big) \\
& = -  \sum_{l=0}^{\infty} z^{lN} N \gamma^{-lN}  \quad \text{(with \cref{eqn:sum_zj}, only $k=lN$ terms left)} \\
& = - \frac{N}{ 1 - z^N/\gamma^N} = Ng\!\left(\Big(\frac{z}{\gamma}\Big)^N\right).
\end{align*}
If $|z| > \gamma$, we can similarly prove the result using $|z_j / z| <1 $.
Finally, since both sides of \cref{eqn:sum} are analytic functions on $\mathbb{C} \backslash  \{ z_j \}_{j=1}^N$
and they agree on $z$ with $|z| \neq \gamma$, by continuity, they must also agree on $z$ with
$|z| = \gamma,\;z \notin  \{ z_j \}_{j=1}^N$. This completes the proof.
\end{proof}

In the following theorem, we derive an analytical expression for the accuracy of approximating $\kappa(x,y)$ by $\tilde{\kappa} (x,y)$.
Without loss of generality, assume $x\ne0$.

\begin{theorem}\label{prop:rel_err}
Suppose $\kappa(x,y)$ in \cref{eqn:kernel} is approximated by $\tilde{\kappa} (x,y)$ in \cref{eqn:kn} which is obtained from the composite trapezoidal rule with \cref{eqn:trapezoidal}. Assume $x\ne0$. Then
\begin{equation}\label{eqn:kn_analytic}
\tilde{\kappa}(x,y) =
 \kappa(x,y)   \left(  1+  \varepsilon(x,y) \right),
\end{equation}
where $\varepsilon(x,y)$ is the relative approximation error
\begin{equation}\label{eqn:rel_err}
\varepsilon(x,y) := \frac{\tilde{\kappa}(x,y) - \kappa(x,y)}{\kappa(x,y)} =  g\left(\Big(\frac{y}{\gamma}\Big)^N\right)  + \sum_{j=0}^{d-1}  \frac{(y-x)^j}{j!} \frac{ \mathrm{d}^j  }{ \mathrm{d} x^j }  g\left(\Big(\frac{\gamma}{x}\Big)^N\right)   .
\end{equation}
\end{theorem}

\begin{proof}
We prove this theorem by induction on $d$. For $d = 1$, substituting \cref{eqn:trapezoidal} into \cref{eqn:kn} yields
\begin{align*}
\tilde{\kappa}  (x, y)  & = \frac{1}{N}  \sum_{j=1}^{N} \frac{z_j}{ (x-z_j)(y - z_j) }  = \frac{1}{N(x-y)} \sum_{j=1}^{N} \frac{(x-z_j)-(y-z_j)}{ (x-z_j)(y - z_j) } z_j\\
& = \frac{1}{N(x-y)} \left(   \sum_{j=1}^{N} \frac{z_j}{y - z_j}  -  \sum_{j=1}^{N} \frac{z_j}{x - z_j} \right) \\
&  = \frac{1}{N(x-y)}  \left(  Ng\left(\Big(\frac{y}{\gamma}\Big)^N\right) - Ng\left(\Big(\frac{x}{\gamma}\Big)^N\right)  \right) \quad \text{(\cref{lem:sum})}  \\
& = \frac{1}{x-y} \left[  1+g\left(\Big(\frac{y}{\gamma}\Big)^N\right) +g\left(\Big(\frac{\gamma}{x}\Big)^N\right)\right].
\end{align*}
Thus, \cref{eqn:kn_analytic} holds for $d = 1$.

Now suppose \cref{eqn:kn_analytic} holds for $d = k$ with $k$ a positive integer. Equating \cref{eqn:kn} and \cref{eqn:kn_analytic} (with $d=k$) and plugging in $\kappa(x,y)$ to get
\begin{equation*}
    \sum_{j=1}^{N} \frac{ \phi_j (z_j,y) }{(x-z_j)^k}   = \frac{1}{(x-y)^k}   \left[  1+  g\left(\Big(\frac{y}{\gamma}\Big)^N\right)  + \sum_{j=0}^{k-1}  \frac{(y-x)^j}{j!} \frac{ \mathrm{d}^j  }{ \mathrm{d} x^j }  g\left(\Big(\frac{\gamma}{x}\Big)^N\right) \right].
\end{equation*}
The derivatives of the left and right-hand sides with respect to $x$ are, respectively,
$- k \sum_{j=1}^{N}  \frac{ \phi_j (z_j,y)}{(x-z_j)^{k+1}}  $
and
\begin{align*}
   &\frac{-k}{(x-y)^{k+1}}   \left[  1+  g\left(\Big(\frac{y}{\gamma}\Big)^N\right)  + \sum_{j=0}^{k-1}  \frac{(y-x)^j}{j!} \frac{ \mathrm{d}^j  }{ \mathrm{d} x^j }  g \left(\Big(\frac{\gamma}{x}\Big)^N\right) \right]  \\
    & \hspace{5pt}+ \frac{1}{(x-y)^{k}}   \left[  \sum_{j=0}^{k-1}  \frac{(y-x)^j}{j!} \frac{ \mathrm{d}^{j+1}  }{ \mathrm{d} x^{j+1} }   g \left(\Big(\frac{\gamma}{x}\Big)^N\right)  -  \sum_{j=1}^{k-1}  \frac{(y-x)^{j-1}}{(j-1)!} \frac{ \mathrm{d}^{j}  }{ \mathrm{d} x^{j} } g \left(\Big(\frac{\gamma}{x}\Big)^N\right) \right] \\
    & \hspace{-5pt}  = \frac{-k}{(x-y)^{k+1}}   \left[  1+  g \left(  \Big(\frac{y}{\gamma}\Big)^N  \right)  + \sum_{j=0}^{k-1}  \frac{(y-x)^j}{j!} \frac{ \mathrm{d}^j  }{ \mathrm{d} x^j }  g \left( \Big(\frac{\gamma}{x}\Big)^N   \right) \right]   \\
    & \hspace{5pt} + \frac{1}{(x-y)^{k}}   \frac{  (y-x)^{k-1}}{ (k-1)! } \frac{ \mathrm{d}^k  }{ \mathrm{d} x^k }  g \left( \Big(\frac{\gamma}{x}\Big)^N  \right) \quad \text{(all terms cancel except for $j=k-1$)} \\
     &  \hspace{-5pt} = \frac{-k}{(x-y)^{k+1}}   \left[  1+  g \left(  \Big(\frac{y}{\gamma}\Big)^N  \right)  + \sum_{j=0}^{k}  \frac{(y-x)^j}{j!} \frac{ \mathrm{d}^j  }{ \mathrm{d} x^j }  g \left( \Big(\frac{\gamma}{x}\Big)^N   \right) \right].
\end{align*}
Thus,
\begin{equation*}
     \sum_{j=1}^{N}  \frac{ \phi (z_j,y) }{(x-z_j)^{k+1}}  = \frac{1}{(x-y)^{k+1}}   \left[  1+  g \left(  \Big(\frac{y}{\gamma}\Big)^N  \right)  + \sum_{j=0}^{k}  \frac{(y-x)^j}{j!} \frac{ \mathrm{d}^j  }{ \mathrm{d} x^j }  g \left( \Big(\frac{\gamma}{x}\Big)^N   \right) \right].
\end{equation*}
That is, \cref{eqn:kn_analytic} holds for $d = k+1$.
By induction, \cref{eqn:kn_analytic}--\cref{eqn:rel_err} are true for any positive integer $d$.
\end{proof}

With the analytical expression \cref{eqn:rel_err} we can give a rigorous upper bound for the approximation error.

\begin{theorem}\label{thm:rel_err_bound}
Suppose $0<|x|<\gamma_1<\gamma<|y|$. With all the assumptions in \cref{prop:rel_err}, there exists a positive integer $N_1$ such that for any $N > N_1$, the approximation error \cref{eqn:rel_err} is bounded by
\begin{equation}\label{eqn:rel_err_bound}
| \varepsilon (x,y ) |  \leq   g\left( \Big|\frac{y}{\gamma}\Big|^N  \right) +  c \, g \!\left( \Big|\frac{\gamma}{x}\Big|^N \right),
\end{equation}
where $c=1$ if $d=1$, and otherwise,
\begin{equation}\label{eq:c}
c = 2+2 \sum_{j=1}^{d-1}  \frac{ [( |y/x|  +1  )N]^j (2d)^{j-1}}{ j!}.
\end{equation}
\end{theorem}

\begin{proof}
    For any positive integer $N$,
    \begin{equation*}
       \left| g\! \left(  \Big(\frac{y}{\gamma}\Big)^N  \right)  \right|  = \frac{1}{ |  (y /\gamma)^N  -1| } \leq \frac{1}{ |y /\gamma|^N -1  } = g\!\left( \Big|\frac{y}{\gamma}\Big|^N \right).
    \end{equation*}
    Thus, we only need to prove the following bound:
    \begin{equation}\label{eqn:rel_err_bound_proof_1}
        \left|  \sum_{j=0}^{d-1}  \frac{(y-x)^j}{j!} \frac{ \mathrm{d}^j  }{ \mathrm{d} x^j }  g \!\left( \Big(\frac{\gamma}{x}\Big)^N   \right) \right|    \leq c\, g\! \left( \Big|\frac{\gamma}{x}\Big|^N \right).
    \end{equation}
    When $d = 1$, it's easy to verify that the above inequality holds for $c = 1$ and any positive integer $N$. We now consider the case when $d \geq 2$.

    It can be verified that, for any positive integer $i$,
    \begin{equation}\label{eqn:rel_err_bound_proof_2}
        \frac{ \mathrm{d} }{ \mathrm{d} x } g^i \left( \Big(\frac{\gamma}{x}\Big)^N   \right) =  \frac{i N}{ x} \left[ g^i \left( \Big(\frac{\gamma}{x}\Big)^N   \right) + g^{i+1} \left( \Big(\frac{\gamma}{x}\Big)^N   \right)     \right],
    \end{equation}
    where $g^i$ denotes function $g$ raised to power $i$. Hence, the derivatives appearing in \cref{eqn:rel_err_bound_proof_1} all have the following form:
    \begin{equation}\label{eqn:rel_err_bound_proof_3}
        \frac{ \mathrm{d}^j  }{ \mathrm{d} x^j }  g \left( \Big(\frac{\gamma}{x}\Big)^N   \right)  =  \frac{1}{ x^j}  \sum_{i = 1}^{j+1} \alpha_i^{(j)} g^i \left( \Big(\frac{\gamma}{x}\Big)^N   \right),
    \end{equation}
    where $\alpha_i^{(j)}$ ($ 1 \leq i \leq j+1$, $ 0 \leq  j \leq d-1 $) are constants.

    We claim that, when $N > d$ and for any $0 \leq j \leq d-1$, there exit constants $\beta^{(j)}$ dependent on $d$ so that
     \[ | \alpha_i^{(j)} | \leq  \beta^{(j)} N^j,\quad 1 \leq i \leq j+1.\]
    This claim can be proved by induction on $j$. It is obviously true when $j=0$, and $\beta^{(0)}=1$ in this case. When $j = 1$, \cref{eqn:rel_err_bound_proof_2} means that
    the claim is true with $\alpha_1^{(1)} = \alpha_2^{(1)} = N$ and $\beta^{(1)} = 1$. Suppose the claim holds for $ j = k$ with $1 \leq k \leq d-2$ (where we also assume $d > 2$, since otherwise the claim is already proved). Then
\begin{align*}
&  \frac{\mathrm{d}^{k+1}}{\mathrm{d}x^{k+1}}g\left(  \Big(\frac{\gamma}{x}\Big)^{N}\right)  =\frac{\mathrm{d}}{\mathrm{d}x}
\left(  \frac{1}{x^{k}}\sum_{i=1}^{k+1}\alpha_{i}^{(k)}g^{i}\left(  \Big(\frac{\gamma}{x}\Big)^{N}\right)  \right)  \\
=\, &  -\frac{k}{x^{k+1}}\sum_{i=1}^{k+1}\alpha_{i}^{(k)}g^{i}\left(
\Big(\frac{\gamma}{x}\Big)^{N}\right)  +\frac{1}{x^{k}}\sum_{i=1}^{k+1}\alpha_{i}^{(k)}\frac{iN}{x}\left[  g^{i}
\left(  \Big(\frac{\gamma}{x}\Big)^{N}\right)  +g^{i+1}\left(  \Big(\frac{\gamma}{x}\Big)^{N}\right)\right]  \\
&  \omit\hfill(\text{by \cref{eqn:rel_err_bound_proof_2}})\\
=\, &  \frac{1}{x^{k+1}}\bigg[(N-k)\alpha_{1}^{(k)}g\left(  \Big(\frac{\gamma}{x}\Big)^{N}\right)  +\sum_{i=2}^{k+1}
\left(  (iN-k)\alpha_{i}^{(k)}+N(i-1)\alpha_{j-1}^{(k)}\right)  g^{i}\left(  \Big(\frac{\gamma}{x}\Big)^{N}\right)  \\
&  \hspace{40pt}+N(k+1)\alpha_{k+1}^{(k)}g^{k+2}\left(  \Big(\frac{\gamma}{x}\Big)^{N}\right)  \bigg].
\end{align*}
    Thus, the coefficients satisfy the following recurrence relation
    \begin{equation*}
        \alpha_i^{(k+1)} = \begin{cases}
            (N -k) \alpha_1^{(k)},\quad  & i = 1, \\
            (iN-k) \alpha_i^{(k)} +  N(i-1)\alpha_{i-1}^{(k)} , \quad & 2 \leq i \leq k+1, \\
            N(k+1) \alpha_{k+1}^{(k)} , \quad & i = k+2.
        \end{cases}
    \end{equation*}
    Therefore, when $N> d$, we can pick (conservatively)
    \begin{equation}\label{eq:crecur}
    \beta^{(k+1)} =  2d \beta^{(k)},
    \end{equation}
    so that $| \alpha_i^{(k+1)} | \leq \beta^{(k+1)} N^{k+1} $. That is, the claim holds for $j = k+1$ and this finishes the induction.

    Now, we go back to prove \cref{eqn:rel_err_bound_proof_1}. By \cref{eqn:rel_err_bound_proof_3},
    \begin{align}\label{eqn:rel_err_bound_proof_4}
         & \left|  \sum_{j=0}^{d-1}  \frac{(y-x)^j}{j!}  \frac{ \mathrm{d}^j  }{ \mathrm{d} x^j }  g \left( \Big(\frac{\gamma}{x}\Big)^N   \right)  \right|  =  \left|  \sum_{j=0}^{d-1} \left[ \frac{(y-x)^j}{j!}   \frac{1}{ x^j}  \sum_{i = 1}^{j+1} \alpha_i^{(j)} g^i \left( \Big(\frac{\gamma}{x}\Big)^N   \right) \right]   \right| \\
        \leq &\sum_{j = 0}^{d - 1} \left[\frac{ ( |y/x|  +1  )^j }{ j!} \sum_{i=1}^{j+1 }  |\alpha_i^{(j)}|\,g^i\! \left(  \Big|\frac{\gamma}{x}\Big|^N   \right)\right]\nonumber
        \leq \sum_{j = 0}^{d - 1} \left[\frac{ ( |y/x|  +1  )^j }{ j!} \beta^{(j)} N^j \sum_{i=1}^{j+1 }  g^i\! \left(  \Big|\frac{\gamma}{x}\Big|^N   \right)\right].\nonumber
    \end{align}
    Set
    \begin{equation}\label{eq:ntilde}
    N_1 = \max\{d, \lceil\log 3/ \log |\gamma_1/x|\rceil \}.
    \end{equation}
    Then for $N > N_1$, $|\gamma/ x|^N>|\gamma_1/x|^N > 3$ and $g \left(  |\gamma /x |^N   \right)  < 1/2$. Thus, for $1\le j\le d-1$,
    \begin{equation*}
        \sum_{i=1}^{j+1 }  g^i \left(  \Big|\frac{\gamma}{x}\Big|^N   \right)  \leq  2 g \left( \Big|\frac{\gamma}{x}\Big|^N \right).
    \end{equation*}
    Continuing on \cref{eqn:rel_err_bound_proof_4}, for $N > N_1$, we get
    \begin{equation}\label{eqn:rel_err_bound_proof_c}
        \left|  \sum_{j=0}^{d-1}  \frac{(y-x)^j}{j!} \frac{ \mathrm{d}^j  }{ \mathrm{d} x^j }  g \left( \Big(\frac{\gamma}{x}\Big)^N   \right)    \right| \leq  c g \left( \Big|\frac{\gamma}{x}\Big|^N \right),\quad\text{with}\quad  c  = 2 \sum_{j=0}^{d-1}  \frac{ ( |y/x|  +1  )^j }{ j!}  \beta^{(j)} N^j.
    \end{equation}
Note that with the way $\beta^{(j)}$ is picked as in \cref{eq:crecur}, $\beta^{(j)}$ satisfies
\[\beta^{(j)}=(2d)^{j-1} \beta^{(1)}=(2d)^{j-1},\ j=1,2,\ldots,d-1.\]
Then $c$ in \cref{eqn:rel_err_bound_proof_c} becomes \cref{eq:c}.
Thus, \cref{eqn:rel_err_bound_proof_1} holds with $c$ in \cref{eq:c}.
\end{proof}

%\begin{remark}\label{rem:n1}
%    In \cref{eq:ntilde}, $N_1$ is independent of $x$ and $y$, which is useful in extending the result to the study of the approximation of $K^{(X,Y)}$ in later sections.
%\end{remark}

The upper bound \cref{eqn:rel_err_bound} in \cref{thm:rel_err_bound}
has two implications.
\begin{itemize}
\item Since $g(|y/\gamma|^N)$ and $g(|\gamma/x|^N)$ decay almost exponentially with $N$ and $c$ is just a polynomial in $N$, $d$, and $|y/x|$ with degrees up to $d-1$, the bound in \cref{eqn:rel_err_bound} decays roughly exponentially as $N$ increases.
\item The bound can help us identify a nearly optimal radius $\gamma$ of the proxy surface $\Gamma$ so as to minimize the error. This is given in the following theorem.
\end{itemize}

\begin{theorem}\label{thm:optimal_gamma}
Suppose $0<|x|<\gamma_1<|y|$ and $\kappa(x,y)$ in \cref{eqn:kernel} is approximated by $\tilde{\kappa} (x,y)$ in \cref{eqn:kn} with \cref{eqn:trapezoidal}.
%With all the assumptions in \cref{thm:rel_err_bound},
If the upper bound in \cref{eqn:rel_err_bound} is viewed as a real function in $\gamma$ on the interval $(|x|,|y|)$, then there exists a positive integer $N_2$ independent of $\gamma$, such that for $N > N_2$,
\begin{enumerate}
    \item this upper bound has a unique minimizer $\gamma^*\in (|x|,|y|)$;
    \item the minimum of this upper bound decays asymptotically as $\mathcal{O} \left( |y/x|^{-N/2}  \right) $.
\end{enumerate}
\end{theorem}

\begin{proof}
    To find the minimizer, we just need to consider the real function
    \begin{equation*}
        h ( t ) =  \frac{1}{ b/t -1 } + \frac{c}{ t /a -1},\quad t \in (a,b),
    \end{equation*}
    where $a = |x|^N$, $b = |y|^N$, and $c$ is either equal to $1$ (for $d=1$) or defined in \cref{eq:c} (for $d\ge 2$). The derivative of the function is
    \begin{equation*}
        h^\prime (t) = \frac{p(t)}{(t-a)^2 (t-b)^2 },\quad\text{with}\quad p(t) =(b-ac) t^2 + 2ab(c-1) t  + ab(a -bc).
    \end{equation*}
    Consider $p(t)$, which is a quadratic polynomial in $t$ with the following properties.
    \begin{itemize}
        \item The coefficient of the second order term is
            \begin{equation*}
                b-ac = |x|^N  \left( |y/x|^N -c   \right).
            \end{equation*}
            Since $c$ is either equal to $1$ (for $d=1$) or a polynomial in $N$, $d$, and $|y/x|$ with degrees up to $d-1$ (for $d\ge 2$),
            there exists $N_2$ larger than $N_1$ in \cref{thm:rel_err_bound} such that $|y/x|^N > c $ for any $N > N_2$. Thus,
            $b-ac > 0$ for $N > N_2$.
        \item The discriminant is $4abc(a-b)^2 >0$.
        \item When evaluated at $t = a$ and $t = b$, the function $p(t)$ gives respectively
\[ p(a) = -ac(a-b)^2 <0,\quad
            p(b) = b(a-b)^2 > 0.\]
    \end{itemize}
    All the properties above combined indicate that $p(t)$ has one root $t_0  \in (a,b)$ and $h^\prime (t)<0$ on $(a,t_0)$
    and $h^\prime (t) > 0$ on $(t_0,b)$.
    Thus, $t_0$ is the only zero of $p(t)$ in $[a,b]$ and $\gamma^* = \sqrt[\leftroot{-3}\uproot{3}N]{t_0}$ is the unique minimizer of the upper bound in \cref{eqn:rel_err_bound}.
The requirements for picking $N_2$ are $N_2 > N_1$ and $|y/x|^{N_2} > c$.
Hence, $N_2$ is independent of $\gamma$.

    To prove the second part of the theorem, we explicitly compute the root $t_0$ of $p(t) = 0$ in $(a,b)$ and
    substitute it into $h(t)$ to get
    \begin{equation*}
        h(t_0) = \frac{ 2 \sqrt{ c b/a } + (c+1) }{ b/a- 1}  = \frac{ 2 \sqrt{c} |y/x|^{N/2} + (c+1)  }{|y/x|^N -1} \sim \mathcal{O} \left( \Big|\frac{y}{x}\Big|^{-N/2}  \right),
    \end{equation*}
    The details involve tedious algebra and are omitted here.
    \end{proof}

In the proof, we can actually find the minimizer but are not
explicitly writing it out. The reason is that the minimizer depends on $x$ and $y$ and
it makes more sense to write a minimizer later when we consider the low-rank approximation of the kernel matrix.
See the next section.

\section{Low-rank approximation accuracy and proxy point selection in the proxy point method for kernel matrices}\label{sec:optimal}

With the kernel $\kappa(x,y)$ in \cref{eqn:kernel} approximated by $\tilde{\kappa}(x,y)$ in \cref{eqn:kn},
a low-rank approximation to $K^{(X,Y)}$ in \cref{eq:k} as follows is obtained:
\begin{equation}\label{eq:lrapprox}
K^{(X,Y)}\approx\tilde{K}^{(X,Y)}:=(\tilde{\kappa}(x,y)_{x\in X,y\in Y})=K^{(X,Z)}\Phi^{(Z,Y)},
\end{equation}
where $\Phi^{(Z,Y)}=(\phi(z,y)_{z\in Z,y\in Y})$.
The analysis in \cref{subsec:error} provides entrywise approximation errors for \cref{eq:lrapprox}
(with implicit dependence on $x$).
Now, we consider normwise approximation errors for $K^{(X,Y)}$ and obtain relative error bounds independent of the specific $x$ and $y$ points.
The error analysis will be further used
to estimate the optimal choice of the radius $\gamma$ for the proxy surface in the low-rank approximation.
We look at the cases $d=1$ and $d\ge2$ separately.

\subsection{The case $d=1$}\label{subsubsec:d1}

In this case, the proof of \cref{prop:rel_err} for $d = 1$ gives an explicit expression for the entrywise approximation error
\begin{equation}\label{eq:ed1}
    \varepsilon(x,y) = g \left(  \Big( \frac{\gamma}{x} \Big)^N  \right)+ g \left(  \Big( \frac{y}{\gamma} \Big)^N  \right).
\end{equation}
We then have the following result on the low-rank approximation error in Frobenius norm.
\begin{proposition}\label{thm:rel_err_bound_d1}
    Suppose $d = 1$ and $\kappa(x,y)$ in \cref{eqn:kernel} is approximated by $\tilde{\kappa} (x,y)$ in \cref{eqn:kn} with \cref{eqn:trapezoidal}. If $0<|x|<\gamma_1<\gamma<\gamma_2<|y|$ for all $x\in X,y\in Y$,
    then for any $N >0$,
    \begin{equation}\label{eqn:rel_err_bound_d1}
        \frac{  \lVert  \tilde{K}^{(X,Y)} - K^{(X,Y)} \rVert_F }{ \lVert K^{(X,Y)} \rVert_F}  \leq g \left(  \Big( \frac{\gamma}{\gamma_1} \Big)^N  \right)+ g \left(  \Big( \frac{\gamma_2}{\gamma} \Big)^N  \right).
    \end{equation}
    Moreover, if the upper bound on the right-hand side is viewed as a function in $\gamma$, it has a unique minimizer $\gamma^* = \sqrt{\gamma_1 \gamma_2}$ and the minimum is $2 g \left(  ( \gamma_2  /  \gamma_1 )^{N/2}  \right) $  which decays asymptotically as $\mathcal{O} \left( |\gamma_2/\gamma_1|^{-N/2}  \right) $.
\end{proposition}
\begin{proof}

The approximation error bound \cref{eqn:rel_err_bound_d1} is a direct application of the entrywise error in \cref{eq:ed1} together
with the fact that $g(t)$ monotonically decreases for $t>1$.

To find the minimizer of the right-hand side of \cref{eqn:rel_err_bound_d1}, we can either follow the proof in \cref{thm:optimal_gamma} or
simply use the following explicit expression:
\begin{align*}
     g \left(  ( \gamma  /  \gamma_1 )^N  \right)  +  g \left(  ( \gamma_2 / \gamma )^N  \right)  &= \frac{1}{ ( \gamma  /  \gamma_1 )^N  -1 } + \frac{1}{ ( \gamma_2 / \gamma )^N  -1 } \\
    & = -1 + \frac{(\gamma_2 / \gamma_1)^N  - 1  }{  (\gamma_2 / \gamma_1)^N  + 1  - \left(  ( \gamma  /  \gamma_1 )^N +   ( \gamma_2 / \gamma )^N   \right) }.
\end{align*}
We just need to minimize $( \gamma  /  \gamma_1 )^N +   ( \gamma_2 / \gamma )^N$, which reaches its minimum at $\gamma^* = \sqrt{\gamma_1 \gamma_2}$.
\end{proof}

\begin{remark}\label{rem:error1}
Although it is not easy to choose $\gamma$ to minimize the approximation error directly,
the minimizer $\gamma^*$ for the bound in \cref{eqn:rel_err_bound_d1} can serve as a reasonable estimate of the minimizer for the error.
These can be seen from an intuitive numerical example below.
In addition, the minimum $2 g \left(  ( \gamma_2  /  \gamma_1 )^{N/2}  \right) $ of the bound in \cref{eqn:rel_err_bound_d1} decays nearly exponentially as $N$ increases. Thus, to reach a relative approximation accuracy $\tau$, we can conveniently decide the number of proxy points:
\[N=\mathcal{O} \left( \frac{\log (1/\tau) }{ \log (\gamma_2 /\gamma_1)} \right).\]
Clearly, $N$ does not depend on the number of points or the geometries of $X,Y$.
It only depends on $\tau$ and $\gamma_2 /\gamma_1$ which indicates the separation of $X$ and $Y$.
This is consistent with the conclusions in the FMM context \cite{sun2001matrix}.
\end{remark}

\begin{example}\label{ex1}
We use an example to illustrate the results in \cref{thm:rel_err_bound_d1} for $d=1$.
The points in $X$ and $Y$ are uniformly chosen from their corresponding regions and are plotted in \cref{fig:bound_d1_xyz}, where
$m=|X| = 200$, $n =|Y|= 300$, $\gamma_1 = 0.5$, $\gamma_2 = 2$, and $\gamma_3 = 5$.
\end{example}

First, we fix the number of proxy points $N = 20$ and let $\gamma$ vary. We plot the
actual error $\mathcal{E}_N(\gamma):=\lVert  \tilde{K}^{(X,Y)} - K^{(X,Y)} \rVert_F / \lVert K^{(X,Y)} \rVert_F$ and the error bound in \cref{eqn:rel_err_bound_d1}.
See \cref{fig:bound_d1_gamma}. We can see that both plots are V-shape lines and the error bound is a close estimate
of the actual error. Moreover, the bound nicely captures the error behavior, and the actual error reaches its minimum almost at the same location
where the error bound is minimized: $\gamma^* = \sqrt{\gamma_1 \gamma_2} = 1$. Thus, $\gamma^*$
is a nice choice to minimize the error. The proxy points $Z$ with radius $\gamma^*$ are plotted in \cref{fig:bound_d1_xyz}.

Then in \cref{fig:bound_d1_n}, we fix $\gamma=\gamma^*$ and let $N$ vary. Again,
the error bound provides a nice estimate for the error. Furthermore, both the error and the bound decay exponentially like $\mathcal{O} \left( |\gamma_2/\gamma_1|^{-N/2}  \right) =\mathcal{O} (2^{-N})$.

\begin{figure}[h]
\centering
\subfloat[Sets $X$ and $Y$ with $\gamma_1 = 0.5$, $\gamma_2 = 2$, $\gamma_3 = 5$ and proxy points $Z$ selected with radius $\gamma^*=1$]{\label{fig:bound_d1_xyz}
\quad\includegraphics[height=1.8in]{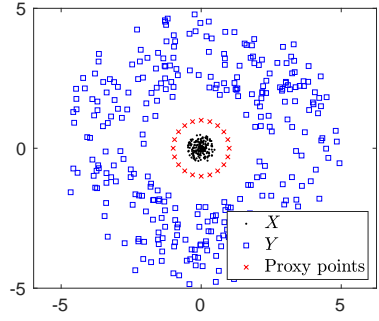}\quad} \\
\subfloat[Varying proxy surface radius $\gamma$]{\label{fig:bound_d1_gamma}\includegraphics[height=1.8in]{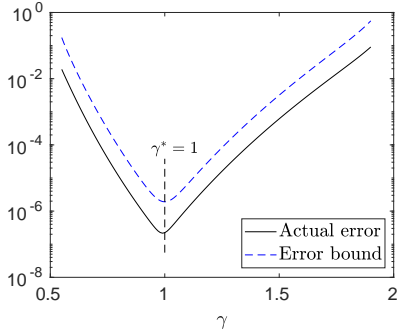}}
\subfloat[Varying number of proxy points $N$]{\label{fig:bound_d1_n}\includegraphics[height=1.8in]{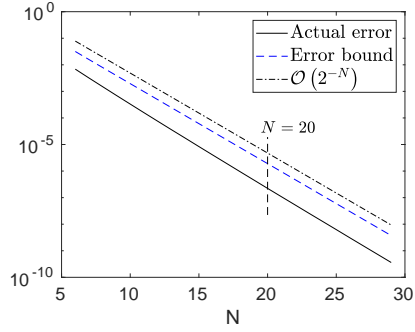}}
\caption{\cref{ex1}: For $d=1$, the selection of the proxy points and the actual relative error $\mathcal{E}_N(\gamma)$ compared with its upper bound in \cref{thm:rel_err_bound_d1} for different $\gamma$ and $N$.}
\label{fig:bound_d1}
\end{figure}

\subsection{The case $d\ge2$}\label{subsubsec:d2}

In this case, there is no simple explicit formula for $\varepsilon(x,y)$ like in \cref{eq:ed1}.
The results in \cref{thm:rel_err_bound,thm:optimal_gamma} cannot be trivially extended to study the normwise error either
since no lower bound is imposed on $|x|$ in $|y/x|$. Nevertheless, we can derive a bound as follows.

\begin{proposition}\label{cor:block_err}
    Suppose $d\ge2$ and $\kappa(x,y)$ in \cref{eqn:kernel} is approximated by $\tilde{\kappa} (x,y)$ in \cref{eqn:kn} with \cref{eqn:trapezoidal}. If $0<|x|<\gamma_1<\gamma<\gamma_2<|y|<\gamma_3$ for all $x\in X,y\in Y$, then there exists a positive integer $N_3$ independent of $\gamma$ such that for $N > N_3$,
    \begin{equation}\label{eqn:block_err_bound}
        \frac{  \lVert  \tilde{K}^{(X,Y)} - K^{(X,Y)} \rVert_F }{ \lVert K^{(X,Y)} \rVert_F} \leq  g\left( \Big(\frac{\gamma_2}{\gamma}\Big)^N  \right) +  \hat{c}\,  g\! \left( \Big(\frac{\gamma}{\gamma_1}\Big)^N \right).
    \end{equation}
    where
    \begin{equation}\label{eq:chat}
    \hat{c}=2+2 \sum_{j=1}^{d-1}  \frac{ [( |\gamma_3/\gamma_1|  +1  )N]^j (2d)^{j-1} }{ j!}.
    \end{equation}
    Moreover,
    if the upper bound in \cref{eqn:block_err_bound} is viewed as a real function in $\gamma$ on the interval $(\gamma_1,\gamma_2)$, then
\begin{enumerate}
    \item this upper bound has a unique minimizer
    \begin{equation}\label{eq:gammamin}
     \gamma^*=\left(\frac{ (\gamma_2^N-\gamma_1^N) \sqrt{(\gamma_1\gamma_2)^N\hat{c}}-(\gamma_1\gamma_2)^N(\hat{c}-1) }{\gamma_2^N - \gamma_1^N\hat{c}}\right)^{1/N}\in(\gamma_1,\gamma_2);
    \end{equation}

    \item the minimum of this upper bound decays asymptotically as $\mathcal{O} \left( |\gamma_2/\gamma_1|^{-N/2}  \right) $.
\end{enumerate}
\end{proposition}

\begin{proof}
Following the proof of \cref{thm:optimal_gamma}, we can set $N_3$ to be the maximum of $N_2$ in \cref{thm:optimal_gamma} for all $x\in X$.
Based on the entrywise error bound in \cref{eqn:rel_err_bound}, we can just show the following inequalities for $N>N_3$ and any $x\in X,y\in Y$:
\[g\left( \Big|\frac{y}{\gamma}\Big|^N \right)< g\left( \Big(\frac{\gamma_2}{\gamma}\Big)^N \right),\quad c  g\left( \Big|\frac{\gamma}{x}\Big|^N \right)<\hat{c} \, g\! \left( \Big(\frac{\gamma}{\gamma_1}\Big)^N \right).\]
The first inequality is obvious. We then focus on the second one.
    Just for the purpose of this proof, we write $c$ in \cref{eq:c} as $c(|x|, |y|)$ to indicate its dependency on $|x|$ and $|y|$.
    $c(|x|, |y|)$ can be viewed as a degree-$(d-1)$ polynomial in $1/|x|$ and $|y|$ with all positive coefficients.

Write
    \begin{equation*}
        c(|x|, |y|) \, g\! \left( \Big|\frac{\gamma}{x}\Big|^N \right)
        =  \left[ c(|x|, |y|) |x|^{d-1}  \right] \left[ g\! \left( \Big|\frac{\gamma}{x}\Big|^N \right) |x|^{1-d} \right].
    \end{equation*}
    The first term $c(|x|, |y|) |x|^{d-1}$ is a polynomial in $|x|$ with all positive coefficients and increases with $|x|$.
    The second term is
    \begin{equation*}
       g\! \left( \Big|\frac{\gamma}{x}\Big|^N \right) |x|^{1-d} = \frac{|x|^{N-d+1}}{ \gamma^N - |x|^{N} }.
    \end{equation*}
    With $N>N_3$, it can be shown that this term is also strictly increasing in $|x|$ for $0<|x| < \gamma_1 < \gamma$. %$\frac{(N-d+1)t^{N-d}\gamma^N+ 2(d-1)t^{2N-1}}{(\gamma^N - t^{N+d-1})^2}$

    Thus for any $x\in X,y\in Y$,
    \[
        c(|x|, |y|) \, g\! \left( \Big|\frac{\gamma}{x}\Big|^N \right) < c(\gamma_1, |y|) \, g\! \left( \Big|\frac{\gamma}{\gamma_1}\Big|^N \right)
        < c(\gamma_1, \gamma_3) \, g\! \left( \Big|\frac{\gamma}{\gamma_1}\Big|^N \right)
        =\hat{c}\, g\! \left( \Big|\frac{\gamma}{\gamma_1}\Big|^N \right),
    \]
    where the constant $\hat{c}$ is defined in \cref{eq:chat} which is $c$ in \cref{eq:c} with $|y/x|$ replaced by $\gamma_3/\gamma_1$.

     The minimizer $\gamma^*$ in \cref{eq:gammamin} for the upper bound is the root of a quadratic polynomial in $(\gamma_1,\gamma_2)$
     and can be obtained following the proof of \cref{thm:optimal_gamma}.
\end{proof}

Based on this corollary, we can draw conclusions similar to those in \cref{rem:error1}.
In addition, although $\gamma_3$ is needed so that $Y$ is on a bounded domain in order to derive the error bound \cref{eqn:block_err_bound},
we believe such an limitation is not needed in practice. In fact, the analytical compression tends to be more accurate when the points $y$
are farther away from the set $X$. Also, if $\gamma_3$ is too large, then we may slightly shift the $x$ points to make sure $|x|$ is larger than
a positive number $\gamma_0$
so as to similarly derive an error bound using $\gamma_0$ instead of $\gamma_3$.

\subsection{A practical method to estimate the optimal radius $\gamma$}\label{subsub:est}

In \cref{thm:rel_err_bound_d1,cor:block_err}, the upper bounds are used to estimate the optimal choice of $\gamma$ for
the radius of the proxy surface. In practice, it is possible that the upper bound may be conservative, especially when $d>1$.
Thus, we also propose the following method to quickly obtain a numerical estimate of the optimal choice.

In \cref{thm:rel_err_bound_d1,cor:block_err}, the optimal $\gamma^*$ is independent of the number of points in $X$ and $Y$ and their distribution.
%was discussed in \cref{rem:n2_gamma}, the optimal $\gamma^*  \in (\gamma_1 , \gamma_2) $ is only dependent on $N$, $d$ and $\gamma_2/\gamma_1$ (replacing $|y/x|$ in the entry-wise analysis) and, most importantly, is independent of the number of points in $X$ and $Y$ and their distribution.
This feature motivates the idea to pick subsets $X_0 \subset \mathcal{D} (0;\gamma_1)$ and $Y_0 \subset \mathcal{A} ( 0; \gamma_2, \gamma_3) $ and
use them to estimate the actual error. That is, we would expect the following two quantities to have similar behaviors when $\gamma$ varies in $(\gamma_1, \gamma_2)$:
\begin{equation}\label{eqn:E_n^0}
	 \mathcal{E}_N^0 (\gamma) :=  \frac{  \lVert  K^{(X_0,Y_0)} - \tilde{K}^{(X_0,Y_0)} \rVert_F }{ \lVert K^{(X_0,Y_0)} \rVert_F},\quad
\mathcal{E}_N (\gamma) :=  \frac{  \lVert  K^{(X,Y)} - \tilde{K}^{(X,Y)} \rVert_F }{ \lVert K^{(X,Y)} \rVert_F}.
\end{equation}
$\mathcal{E}_N^0 (\gamma) $ can be used as an estimator of the actual approximation error $\mathcal{E}_N (\gamma) $.
Note that $K^{(X_0,Y_0)}$ and $\tilde{K}^{(X_0,Y_0)}$ are computable through \cref{eqn:kernel} and \cref{eqn:kn}, respectively, so $\mathcal{E}_N^0 (\gamma)$ can be computed explicitly, and the cost is extremely small if $|X_0| \ll |X|$ and $|Y_0| \ll |Y|$.

Note that in rank-structured matrix computations, often an admissible condition or separation parameter is prespecified for the compression of multiple off-diagonal
blocks. In the case of kernel matrices, it means that the process of estimating the optimal $\gamma$ needs to be run only once and can then be used
in multiple compression steps.

\begin{example}\label{ex2}
We use an example to demonstrate the numerical selection of the optimal $\gamma$. Consider $d=2,3$ and the two sets $X$ and $Y$ in \cref{ex1}
with the same values $\gamma_1,\gamma_2,\gamma_3$ (see \cref{fig:bound_d1_xyz}). Fix $N = 30$.
\end{example}

For the sets $X_0$ and $Y_0$ we choose, we set $l = |X_0| = |Y_0|$ to be $1$, $2$, or $3$. We make sure $x = \gamma_1$ and $y=\gamma_2$ as points of $\mathbb{C}$ are always in $X_0$ and $Y_0$, respectively. These two boundary points correspond to the worst case scenarios of the error bound developed before. Thus, $\mathcal{E}_N^0 (\gamma) $ is more likely to capture the behavior of $\mathcal{E}_N (\gamma)$. Any additional points in $X_0$ are uniformly distributed in the circle $\mathcal{C} (0;\gamma_1)$ and any additional points in $Y_0$ are uniformly distributed in $\mathcal{C} (0;\gamma_2)$.
\begin{figure}[tbhp]
\centering
\subfloat[$d=2$]{\label{fig:d2_1}\includegraphics[height=1.8in]{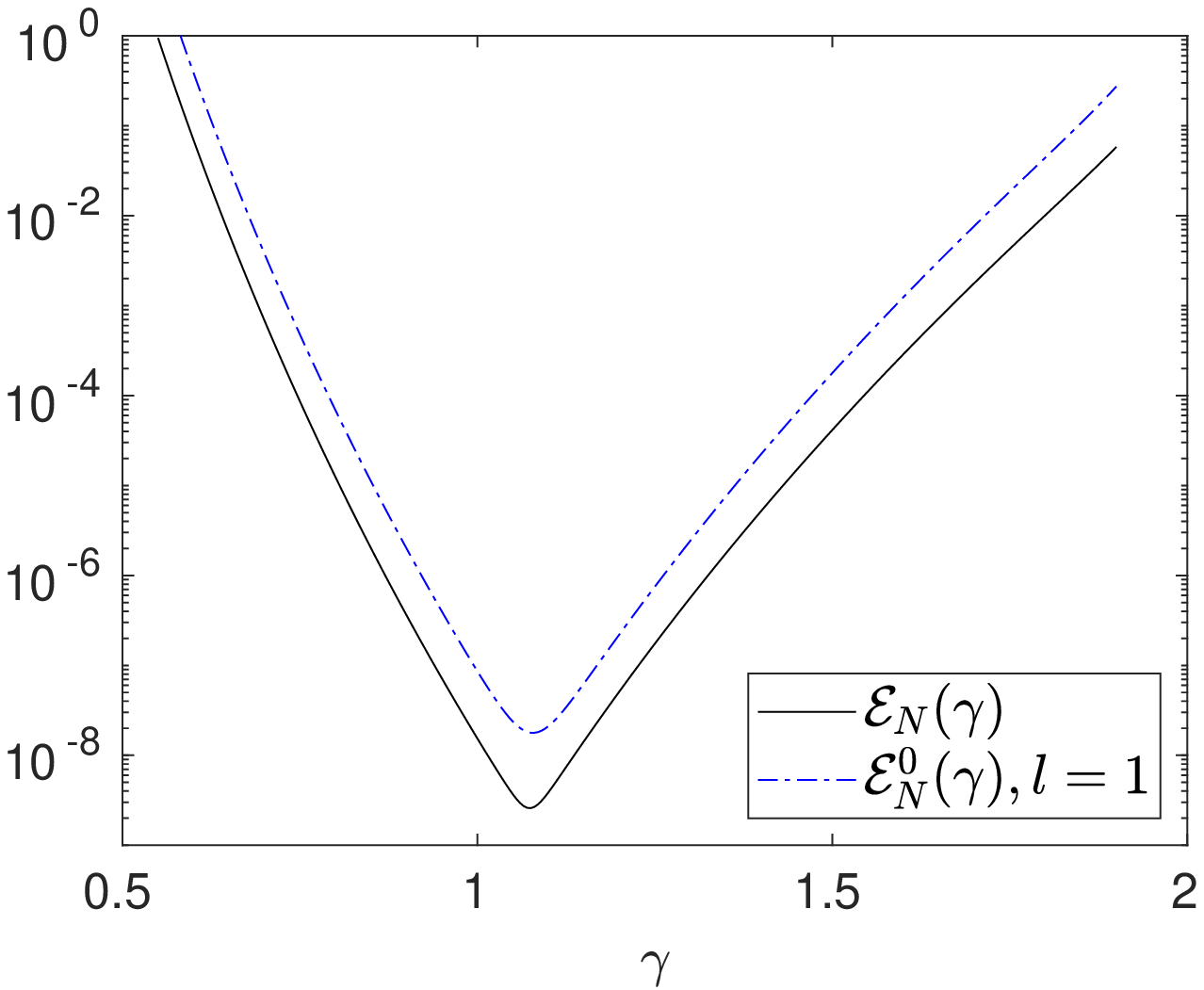}}\quad
\subfloat[$d=3$]{\label{fig:d3_1}\includegraphics[height=1.8in]{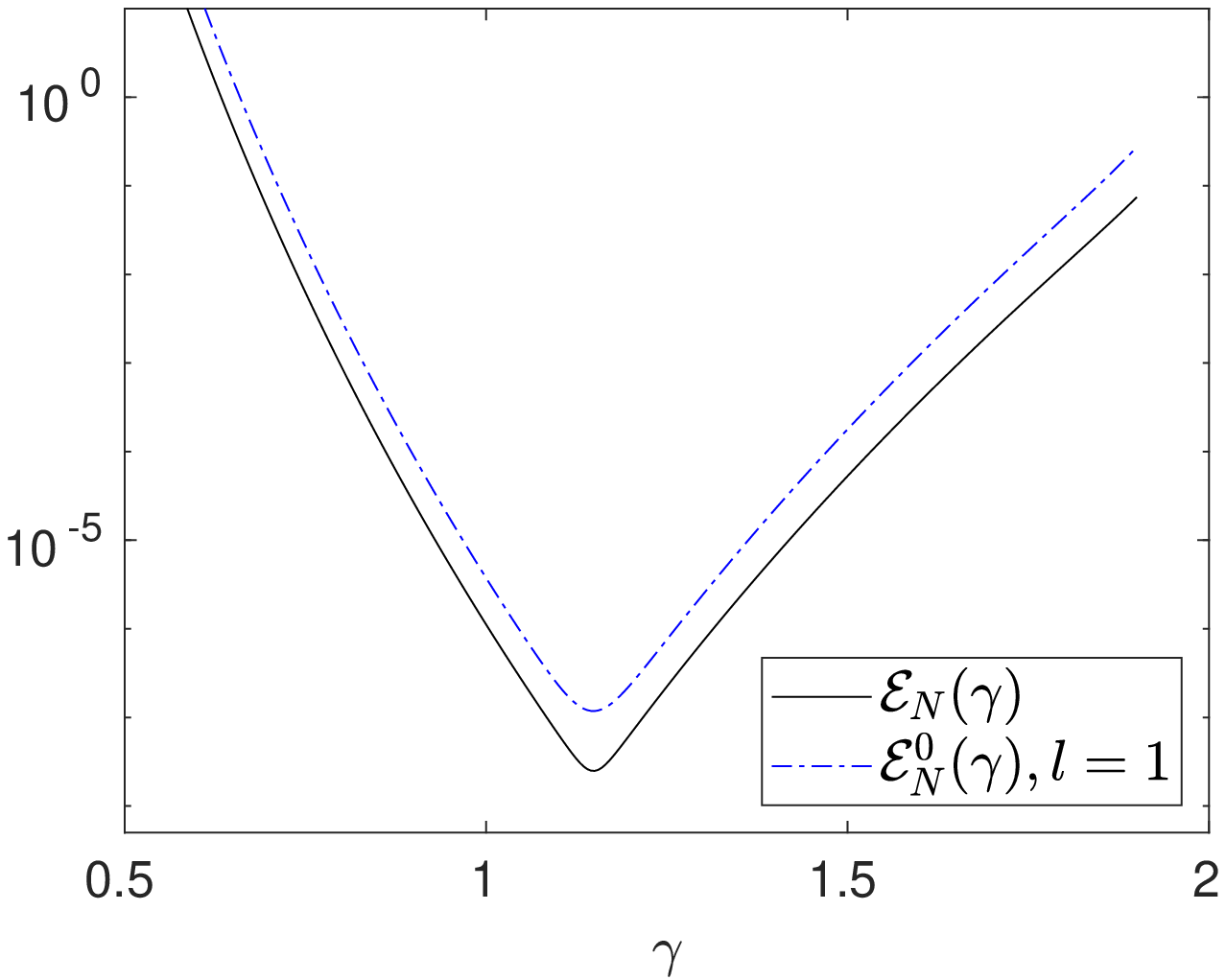}}\\
\subfloat[$d=2$, zoomed in around the critical point]{\label{fig:d2_2}\includegraphics[height=1.8in]{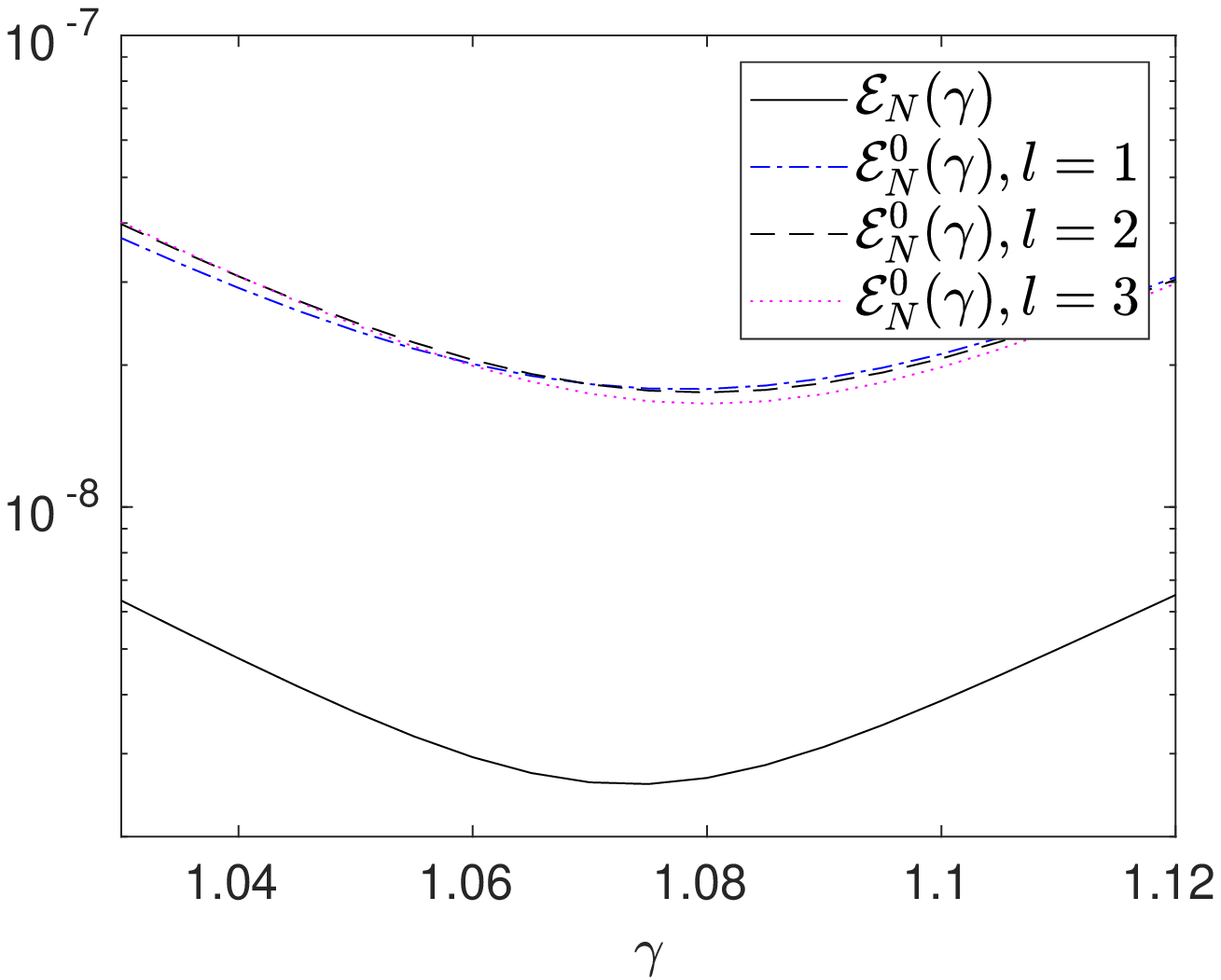}}\quad
\subfloat[$d=3$, zoomed in around the critical point]{\label{fig:d3_2}\includegraphics[height=1.8in]{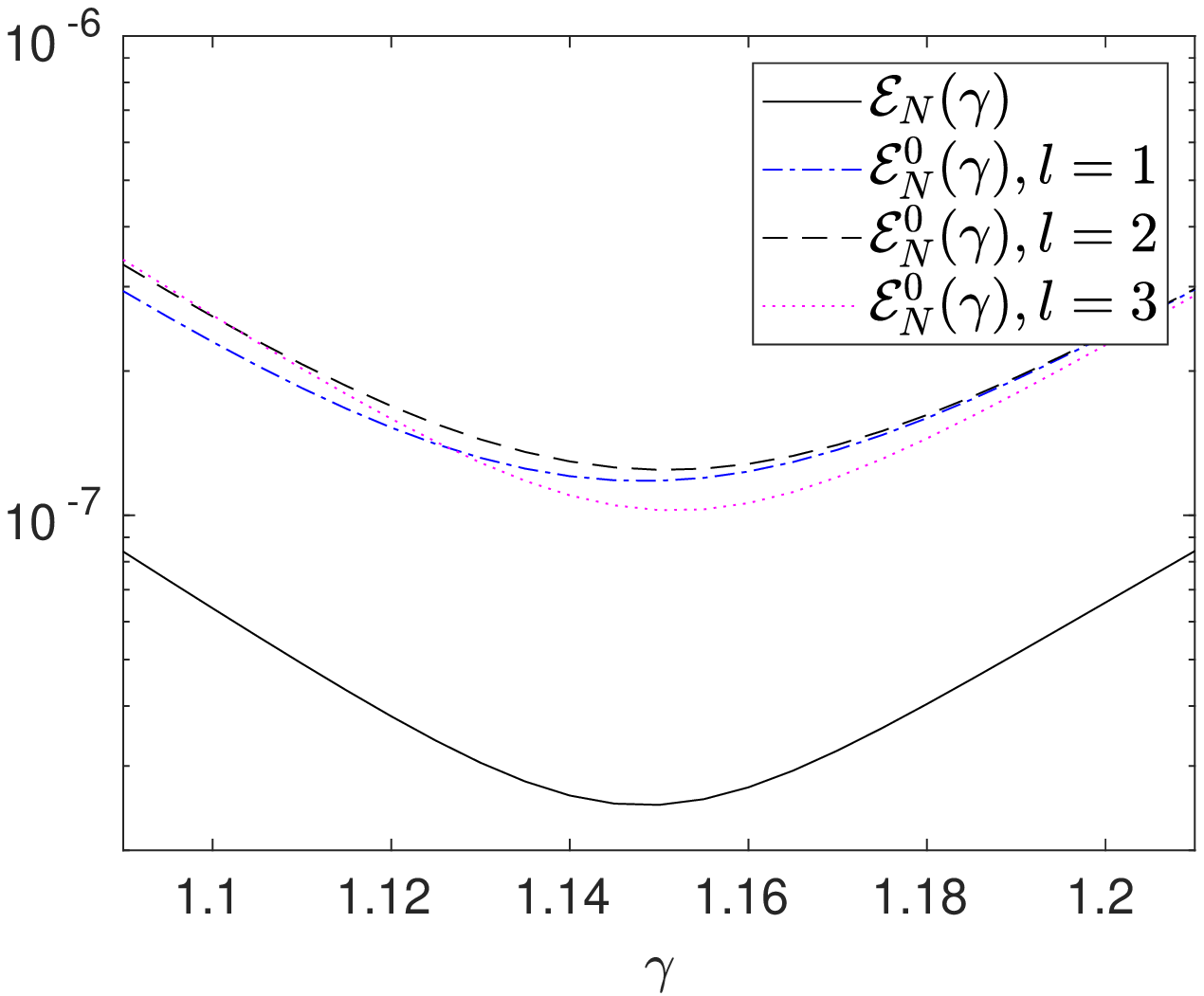}}
\caption{\cref{ex2}: For $d=2$ and $3$, how the estimator $\mathcal{E}_N^0 (\gamma)$ with $l = 1,2,3$ compare with
the actual error $\mathcal{E}_N(\gamma)$.}
\label{fig:bound_d23}
\end{figure}

With $l=1$, both $\mathcal{E}_N(\gamma) $ and $\mathcal{E}_N^0 (\gamma)$ are plotted. See \cref{fig:d2_1,fig:d3_1} for $d=2$ and $3$, respectively.
We can see that $\mathcal{E}_N^0 (\gamma)$ already gives a good estimate of the behavior of $\mathcal{E}_N (\gamma)$ for both cases. Then in \cref{fig:d2_2,fig:d3_2}
we plot $\mathcal{E}_N^0 (\gamma)$ for $l=1,2, 3$ and zoom in at around the minimum since they almost coincide with each other away from the minimum.
The minimums of the three cases are very close to each other, which indicates that $l=1$ suffices to give a reliable
estimate of the optimal radius in practice.

\section{Low-rank approximation accuracy in hybrid compression and representative point selection}
\label{sec:hybrid}

The analytical compression in \cref{sec:optimal} can serve as a preliminary low-rank approximation, which
is typically followed by an algebraic compression step to get a more
compact low-rank approximation. In this section, we analyze the approximation error of such hybrid (analytical/algebraic) compression
applied to $K^{(X,Y)}$.

Suppose $m=|X|$ and $n=|Y|$ are sufficiently large and $N=|Z|$ is fixed. With the preliminary low-rank
approximation in \cref{eq:lrapprox}, since $K^{(X,Z)}$ has a much smaller column size than $K^{(X,Y)}$, it becomes
practical to apply an SRRQR factorization to $K^{(X,Z)}$ to obtain the following approximation:
\begin{equation}\label{eq:uie}
K^{(X,Z)}\approx UK^{(X,Z)}|_J,\quad\text{with}\quad U=  P \left(\begin{array}[c]{c}
I \\ E
\end{array}\right),
\end{equation}
where $P$ is a permutation matrix so that $K^{(X,Z)}|_{J}$ is formed by selected rows of
$K^{(X,Z)}$ with the row index set $J$. $J$ essentially corresponds to a subset $\hat{X}\subset X$ and $\hat{X}$ can be referred to as a set of \textit{representative points} of $X$ so that $K^{(\hat{X},Z)}\equiv K^{(X,Z)}|_{J}$.
\cref{eq:uie} is an interpolative decomposition of $K^{(X,Z)}$.
It is also referred to as a structure-preserving rank-revealing (SPRR) factorization in \cite{toeprs} since $K^{(\hat{X},Z)}$
is a submatrix of $K^{(X,Z)}$.

Although $U$ generally does not have orthonormal columns, the SRRQR factorization keeps its norm under control in the sense that entries of $E$
have magnitudes bounded by a number $e$ (e.g., $2$ or $\sqrt{N}$). See \cite{gu1996efficient} for details.
%Typically, the process costs $\mathcal{O}(m mr)$ if $\hat{X}=r$.

We then have
\begin{subequations}\label{eq:hybrid}
\begin{align}
\hskip15mm K^{(X,Y)}& \approx \tilde{K}^{(X,Y)}=K^{(X,Z)}\Phi ^{(Z,Y)}
&\hfill \text{(by \cref{eq:lrapprox})} \label{eq:hybrid1}\\
& \approx UK^{(\hat{X},Z)}\Phi ^{(Z,Y)} & \hfill \text{(by \cref{eq:uie})}  \label{eq:hybrid2} \\
& =U\tilde{K}^{(\hat{X},Y)}\approx UK^{(\hat{X},Y)}, & \,\text{(by \cref{eqn:kn} and similar to \cref{eq:lrapprox})}\label{eq:hybrid3}
\end{align}
\end{subequations}
which is an SPRR factorization of $K^{(X,Y)}$.

Similarly, an SRRQR factorization can further be applied to $K^{(\hat{X},Y)}$ to produce
\begin{equation}
K^{(\hat{X},Y)}\approx K^{(\hat{X},\hat{Y})}V^T,\quad\text{with}\quad V=Q\left(\begin{array}[c]{c}
I \\ F
\end{array}\right),\label{eq:vif}
\end{equation}
where $Q$ is a permutation matrix and $\hat{Y}\subset Y$.
The approximation \cref{eq:hybrid} together with \cref{eq:vif}
essentially enables us to quickly to select representative points from both $X$ and $Y$. In another
word, we have a skeleton factorization of $K^{(X,Y)}$ as
\begin{equation}\label{eq:skel}
K^{(X,Y)}\approx UK^{(\hat{X},\hat{Y})}V^T.
\end{equation}

Note that computing an SPRR or skeleton factorization for $K^{(X,Y)}$ directly (or to find a submatrix $K^{(\hat{X},\hat{Y})}$ with the largest ``volume" \cite{gor01,tyr96}) is typically prohibitively expensive for large $m$ and $n$.
Here, the proxy point method substantially reduces the cost. In fact,
\cref{eq:hybrid1,eq:hybrid3} are done analytically with no compression cost. Only the SRRQR factorizations of skinny matrices
($K^{(X,Z)}$ and/or $K^{(\hat{X},Y)}$)
are needed. The total compression cost is $\mathcal{O}(mN r)$ for \cref{eq:hybrid} or $\mathcal{O}(mN r+nr^2)$ for \cref{eq:skel} instead of $\mathcal{O}(mnr)$, where $r=|\hat{X}|\ge|\hat{Y}|$. As we have discussed before, $N$ is only a constant independent of $m$ and $n$. Thus,
this procedure is significantly more efficient than applying SRRQR factorizations directly to the original kernel matrix.

The next theorem concerns the approximation error of the hybrid compression via either \cref{eq:hybrid} or \cref{eq:skel}.

\begin{theorem}\label{thm:hybrid_err}
Suppose $0<|x|<\gamma_1<\gamma<\gamma_2<|y|<\gamma_3$ for any $x\in X,y\in Y$ and the $N$ proxy points in $Z$ are
located on the proxy surface with radius $\gamma^*$.
Let $r=|X|$ and let the relative tolerance in the kernel approximation be
$\tau_1$ (i.e., $|\varepsilon(x,y)|<\tau_1$ for $\varepsilon(x,y)$ in \cref{eqn:kn_analytic}) and the
relative approximation tolerance (in Frobenius norm) in the SRRQR factorizations \cref{eq:uie,eq:vif} be $\tau_2$.
Assume the entries of $E$ in \cref{eq:uie} and $F$ in \cref{eq:vif} have magnitudes bounded by $e$.
Then the approximation of $K^{(X,Y)}$ by \cref{eq:hybrid} satisfies
\begin{equation}
\frac{\lVert K^{(X,Y)}-UK^{(\hat{X},Y)}\lVert _{F}}{\lVert K^{(X,Y)}\lVert _{F}}
< s_{1}\tau_{1}+s_{2}\tau_{2},  \label{eqn:hybrid_err}
\end{equation}
where
\begin{align*}
& s_1 = 1 +\sqrt{r + (m-r) r e^2} \sqrt{1 - \frac{ (m-r) (\gamma_2- \gamma_1)^{2d} }{m (\gamma_1 + \gamma_3)^{2d}} },\quad
s_2 = \frac{ \gamma^* (\gamma_1+\gamma_3)^d}{ (\gamma_2 - \gamma^*) (\gamma^*-\gamma_1)^d}   .
\end{align*}
The approximation of $K^{(X,Y)}$ by \cref{eq:skel} satisfies
\begin{equation}
\frac{\lVert K^{(X,Y)}-UK^{(\hat{X},\hat{Y})}V^{T}\lVert _{F}}{\lVert
K^{(X,Y)}\lVert _{F}}<s_{1}\tau _{1}+\tilde{s}_2\tau _{2},\label{eq:skelerr}
\end{equation}
where $\tilde{s}_2=s_{2}+s_{1}-1$.
\end{theorem}

\begin{proof}
The following inequalities for $x\in X,y\in Y,z\in Z$ will be useful in the proof:
\begin{gather}
|\phi (z,y)|< \frac{\gamma ^{\ast }}{N(\gamma _{2}-\gamma ^{\ast })},
\label{eq:phizy} \\
|\kappa (x,z)| <\frac{1}{(\gamma ^{\ast }-\gamma _{1})^{d}},
\label{eq:kxz} \\
\frac{1}{(\gamma _{1}+\gamma _{3})^{d}} <|\kappa (x,y)|<\frac{1}{(\gamma_{2}-\gamma _{1})^{d}}.  \label{eq:kxy}
\end{gather}

Note that
\begin{align}
& \lVert K^{(X,Y)}  - U K^{(\hat{X},Y)} \lVert_F \label{eqn:hybrid_err_proof} \\
\leq\ & \lVert  K^{(X,Y)}  -  \tilde{K}^{(X,Y)}  \rVert_F + \lVert  \tilde{K}^{(X,Y)}  - U  K^{(\hat{X},Y)}  \rVert_F \nonumber\\
\leq\  & \lVert K^{(X,Y)}  -  \tilde{K}^{(X,Y)} \rVert_F +  \lVert \tilde{K}^{(X,Y)} - U\tilde{K}^{(\hat{X},Y)} \rVert_F
 + \lVert  U \tilde{K}^{(\hat{X},Y)} -U  K^{(\hat{X},Y)}\rVert_F \nonumber\\
=\  & \lVert K^{(X,Y)}  -  \tilde{K}^{(X,Y)} \rVert_F +\lVert  K^{(X,Z)}\Phi^{(Z,Y)} - U K^{(\hat{X}, Z)}\Phi^{(Z,Y)} \rVert_F \nonumber\\
& \quad +  \lVert  U \tilde{K}^{(\hat{X},Y)} -U  K^{(\hat{X},Y)}\rVert_F \hskip20mm\text{(by \cref{eq:hybrid1,eq:hybrid2,eq:hybrid3})}\nonumber\\
\leq\  & \lVert  K^{(X,Y)}  -  \tilde{K}^{(X,Y)}  \rVert_F  + \lVert K^{(X,Z)} -  U K^{(\hat{X}, Z)} \rVert_F \lVert \Phi^{(Z,Y)} \rVert_F \nonumber\\
& \quad+  \lVert U \rVert_F \lVert K^{(\hat{X},Y)} - \tilde{K}^{(\hat{X},Y)} \rVert_F.\nonumber
\end{align}
Now, we derive upper bounds separately for the three terms in the last step
above.

The first term is the approximation error for the original kernel
matrix from the proxy point method. Then%\cref{thm:rel_err_bound_d1} (for $d=1$)\ or \cref{cor:block_err} (for $d\geq 2$) means
\begin{equation}
\lVert K^{(X,Y)}-\tilde{K}^{(X,Y)}\rVert _{F}\leq \tau _{1}\lVert
K^{(X,Y)}\rVert _{F}.  \label{eq:hybrid_err_proof1}
\end{equation}

Next, from the SPRR factorization of $K^{(X,Z)}$,
\[
\lVert K^{(X,Z)}-UK^{(\hat{X},Z)}\rVert _{F}\lVert \Phi
^{(Z,Y)}\rVert _{F}
\leq \tau _{2}\lVert {K}^{(X,Z)}\rVert _{F}\lVert \Phi ^{(Z,Y)}\rVert _{F}.
\]
\cref{eq:phizy} means
\begin{equation*}
\lVert \Phi ^{(Z,Y)}\rVert _{F}<\sqrt{Nn}\frac{\gamma ^{\ast }}{N(\gamma
_{2}-\gamma ^{\ast })}=\sqrt{\frac{n}{N}}\frac{\gamma ^{\ast }}{\gamma
_{2}-\gamma ^{\ast }}.
\end{equation*}(\ref{eq:kxz}) and (\ref{eq:kxy}) mean\begin{equation*}
\frac{\lVert K^{(X,Z)}\rVert _{F}^{2}}{\lVert K^{(X,Y)}\rVert _{F}^{2}}<\frac{mN/(\gamma ^{\ast }-\gamma _{1})^{2d}}{mn/(\gamma _{1}+\gamma
_{3})^{2d}}=\frac{N}{n}\frac{(\gamma _{1}+\gamma _{3})^{2d}}{(\gamma ^{\ast
}-\gamma _{1})^{2d}}.
\end{equation*}
Then
\begin{align}
\lVert {K}^{(X,Z)}-UK^{(\hat{X},Z)}\rVert _{F}\lVert \Phi^{(Z,Y)}\rVert _{F}
&<\tau _{2}\sqrt{\frac{n}{N}}\frac{\gamma ^{\ast }}{\gamma _{2}-\gamma
^{\ast }}\lVert {K}^{(X,Z)}\rVert _{F} \label{eq:hybrid_err_proof2} \\
&<\tau _{2}\frac{\gamma ^{\ast }(\gamma _{1}+\gamma _{3})^{d}}{(\gamma
_{2}-\gamma ^{\ast })(\gamma ^{\ast }-\gamma _{1})^{d}}\lVert {K}^{(X,Y)}\rVert _{F}.\nonumber
\end{align}

Thirdly,
\begin{gather*}
\lVert U \rVert_F = \left\lVert P \begin{pmatrix} I \\ E   \end{pmatrix}  \right\rVert_F= \left\lVert \begin{pmatrix} I \\ E   \end{pmatrix}  \right\rVert_F
\leq \sqrt{r + (m-r) r e^2},\\
\lVert K^{(\hat{X}, Y)} - \tilde{K}^{(\hat{X}, Y)} \rVert_F \leq  \tau_1 \lVert K^{(\hat{X}, Y)} \rVert_F.
\end{gather*}
According to \cref{eq:kxy},
\[
 \frac{\lVert K^{(\hat{X},Y)} \rVert_F^2}{\lVert K^{(X,Y)} \rVert_F^2}  =  1 - \frac{\lVert K^{(X \backslash \hat{X},Y)}  \rVert_F^2}{\lVert K^{(X,Y)} \rVert_F^2}  \leq   1 - \frac{(m\!-\!r) n /{(\gamma_1\!+\!\gamma_3)^{2d}} }{m n /{(\gamma_2-\gamma_1)^{2d}}}   =   1 - \frac{ (m\!-\!r) (\gamma_2\!-\! \gamma_1)^{2d} }{m (\gamma_1 + \gamma_3)^{2d}} .
\]
Then
\begin{align}\label{eq:hybrid_err_proof3}
& \lVert U \rVert_F \lVert K^{(\hat{X},Y)} - \tilde{K}^{(\hat{X},Y)} \rVert_F  \\  \leq \ &\tau_1 \sqrt{r + (m-r) r e^2} \sqrt{1 - \frac{ (m-r) (\gamma_2- \gamma_1)^{2d} }{m (\gamma_1 + \gamma_3)^{2d}} } \lVert K^{(X,Y)} \rVert_F.\nonumber
\end{align}

Combining the results \cref{eq:hybrid_err_proof1,eq:hybrid_err_proof2,eq:hybrid_err_proof3} from the four steps above yields \cref{eqn:hybrid_err}.
To show \cref{eq:skelerr}, we use the following inequality:
\begin{align*}
&\lVert K^{(X,Y)}-UK^{(\hat{X},\hat{Y})}V^{T}\lVert _{F}\\
\leq & \ \lVert
K^{(X,Y)}-\tilde{K}^{(X,Y)}\rVert _{F}+\lVert {K}^{(X,Z)}\Phi ^{(Z,Y)}-UK^{(\hat{X},Z)}\Phi ^{(Z,Y)}\rVert _{F} \\
& \ +\lVert U\tilde{K}^{(\hat{X},Y)}-UK^{(\hat{X},Y)}\rVert _{F}+\lVert UK^{(\hat{X},Y)}-UK^{(\hat{X},\hat{Y})}V^{T}\rVert _{F}.
\end{align*}
Then the proof can proceed similarly.
\end{proof}

If $e$ in SRRQR factorizations is a constant, with fixed $N$, the two constants in \cref{eqn:hybrid_err} scale roughly as $s_1 = \mathcal{O} ( \sqrt{m})$ and $s_2 =\mathcal{O} (1)$. Moreover,
once the annulus region $\mathcal{A} ( 0; \gamma_2, \gamma_3)$ is fixed, the set $Y$ is completely irrelevant to the algorithm for obtaining the approximation \cref{eq:hybrid} and the error bound \cref{eqn:hybrid_err}. The column basis matrix $U$ and the set $\hat{X}$ of
representative points can be obtained with only the set $X$, and the error analysis in \cref{eqn:hybrid_err}
applies to any set $Y$ in $\mathcal{A} ( 0; \gamma_2, \gamma_3)$.

\begin{remark}
Note that our error analyses in the previous section and this section are not necessarily restricted to the particular kernel like
in \cref{eqn:kernel} or the proxy point approximation method. In fact, the error bounds
can be easily modified for more general kernels and/or with other approximation methods as long as
a relative error bound for the kernel function approximation is available. This bound is $\tau_1$ in
\cref{thm:hybrid_err}.
\end{remark}

We then use a comprehensive example to show the accuracies of the analytical compression and the hybrid compression, as well as the selections of
the proxy points and the representative points.

\begin{example}\label{ex3}
We generate a triangular finite element mesh on a rectangle domain $[0,2]\times[0,1]$ based on the package MESHPART \cite{meshpart}.
The two sets of points $X$ and $Y$ are the mesh points as shown in \cref{fig:mesh}, where $|X| = 821$, $|Y| = 4125$, $\gamma_1 = 0.3$, and $\gamma_2 = 0.45$.
We compute the low-rank approximation in \cref{eq:hybrid} and report the relative errors in the
analytical compression step and the hybrid low-rank approximation respectively:
\begin{equation*}
	\mathcal{E}_N (\gamma)=\frac{\|  K^{(X,Y)} -  \tilde{K}^{(X,Y)}  \|_F}{\| K^{(X,Y)} \|_F}, \quad
\mathcal{R}_N (\gamma)=\frac{\|  K^{(X,Y)} -  U K^{(\hat{X},Y)}  \|_F}{\| K^{(X,Y)} \|_F}.
\end{equation*}
%A fixed tolerance for SRRQR/ID is used for all tests.
\begin{figure}[tpbh]
\centering
\includegraphics[height=2.5in]{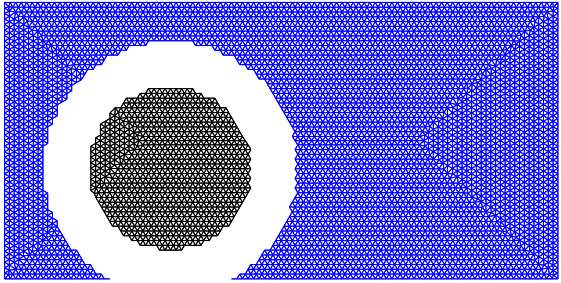}\vspace{-5mm}
\caption{\cref{ex3}: Sets $X$ and $Y$ in the mesh, where the image is based on the package MESHPART \cite{meshpart}.}
\label{fig:mesh}
\end{figure}
\end{example}

In the first set of tests, the number of proxy points $N$ is chosen to reach a relative tolerance $\tau_1=10 \varepsilon_{\text{mach}}$
in the proxy point method, where $\varepsilon_{\text{mach}}$ is the machine precision. (Note that $\tau_1$ is the tolerance for approximating
$\kappa(x,y)$, and the actual computed Frobenius-norm matrix approximation error $\mathcal{E}_N (\gamma)$ may be slightly larger due to
floating point errors.)

We vary the radius $\gamma$ for the proxy surface between $\gamma_1$ and $\gamma_2$. For $d=1,2,3,4$,
$\mathcal{E}_N (\gamma)$ and $\mathcal{R}_N (\gamma)$ are shown in \cref{fig:ex3_1}.
In practice, we can use the method in \cref{subsub:est} to obtain an approximate optimal radius $\tilde{\gamma}^*$.
To show that $\tilde{\gamma}^*$ is very close to the actual optimal radius, we can look at
\cref{fig:test_1_1} for $d=1$. Here, $N = 169$ and $\tilde{\gamma}^*=0.3675$ which is very close
to the actual optimal radius $0.3678$. In addition, the error bound in \cref{thm:rel_err_bound_d1}
can be used to provide another estimate $\sqrt{\gamma_1 \gamma_2} = 0.3674$. Both estimates are very
close to the actual minimizer, which indicates the effectiveness
of the error analysis and the minimizer estimations. When $\gamma = \tilde{\gamma}^*$, we
have $\mathcal{E}_N (\gamma)=3.2106E-16$ and $\mathcal{R}_N (\gamma)=1.1008E-15$, and
the numerical rank resulting from the hybrid compression is $78$. The numerical rank produced by SVD under a similar relative error is $68$.
\begin{figure}[tbhp]
\centering
\subfloat[$d=1$]{\label{fig:test_1_1}\includegraphics[scale=.43]{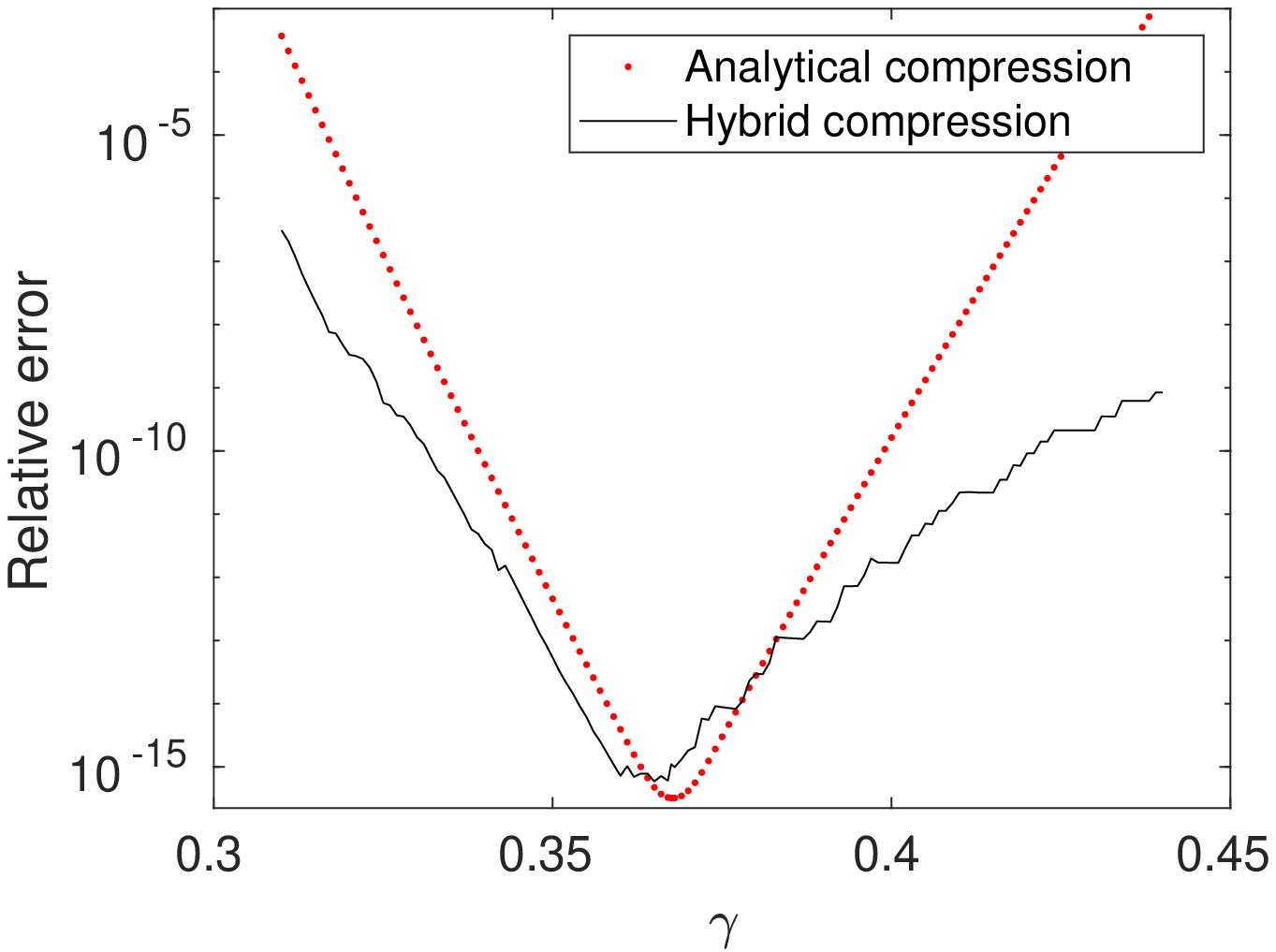}}
\subfloat[$d=2$]{\label{fig:test_1_2}\includegraphics[scale=.43]{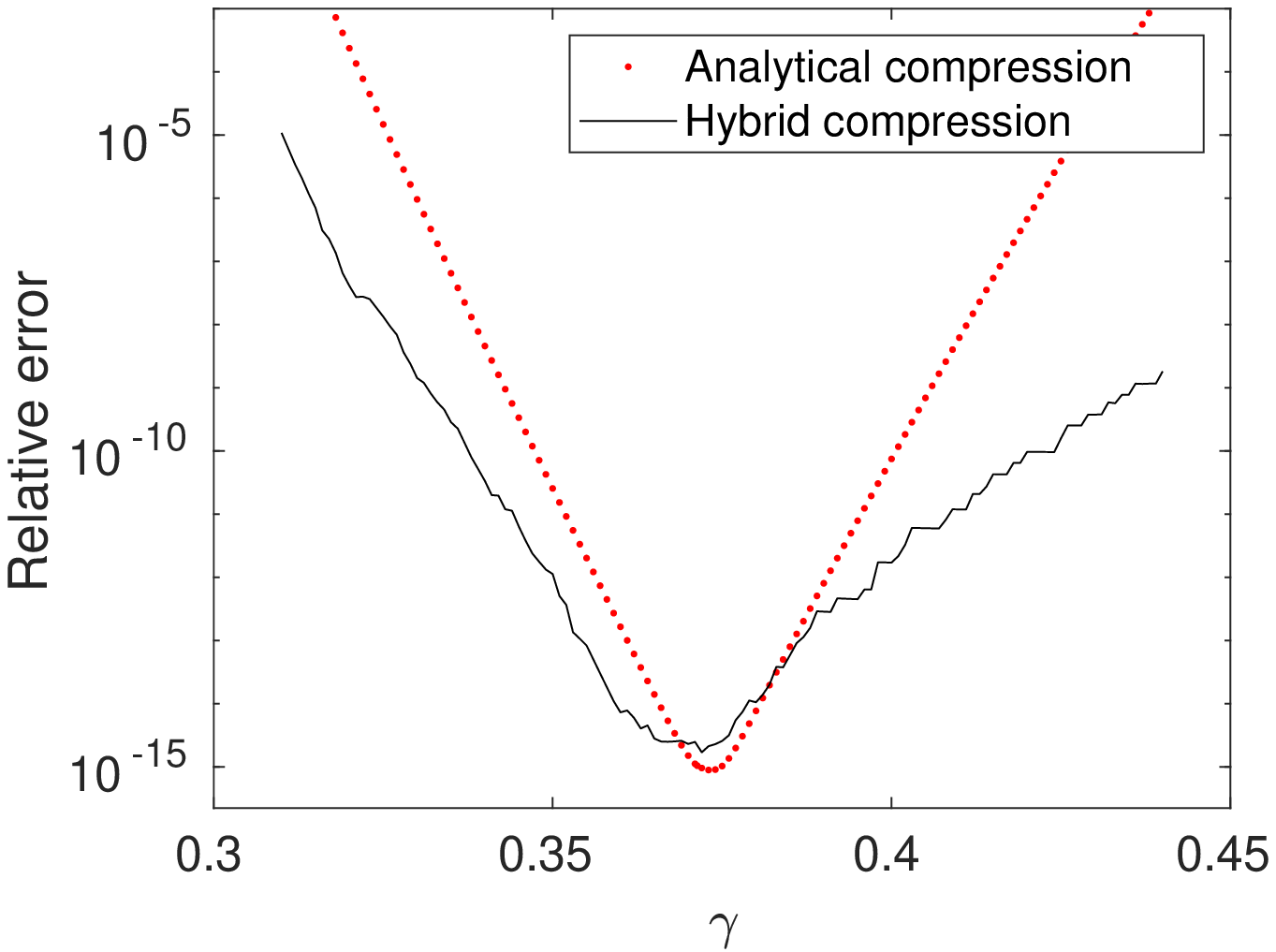}} \\
\subfloat[$d=3$]{\label{fig:test_1_3}\includegraphics[scale=.43]{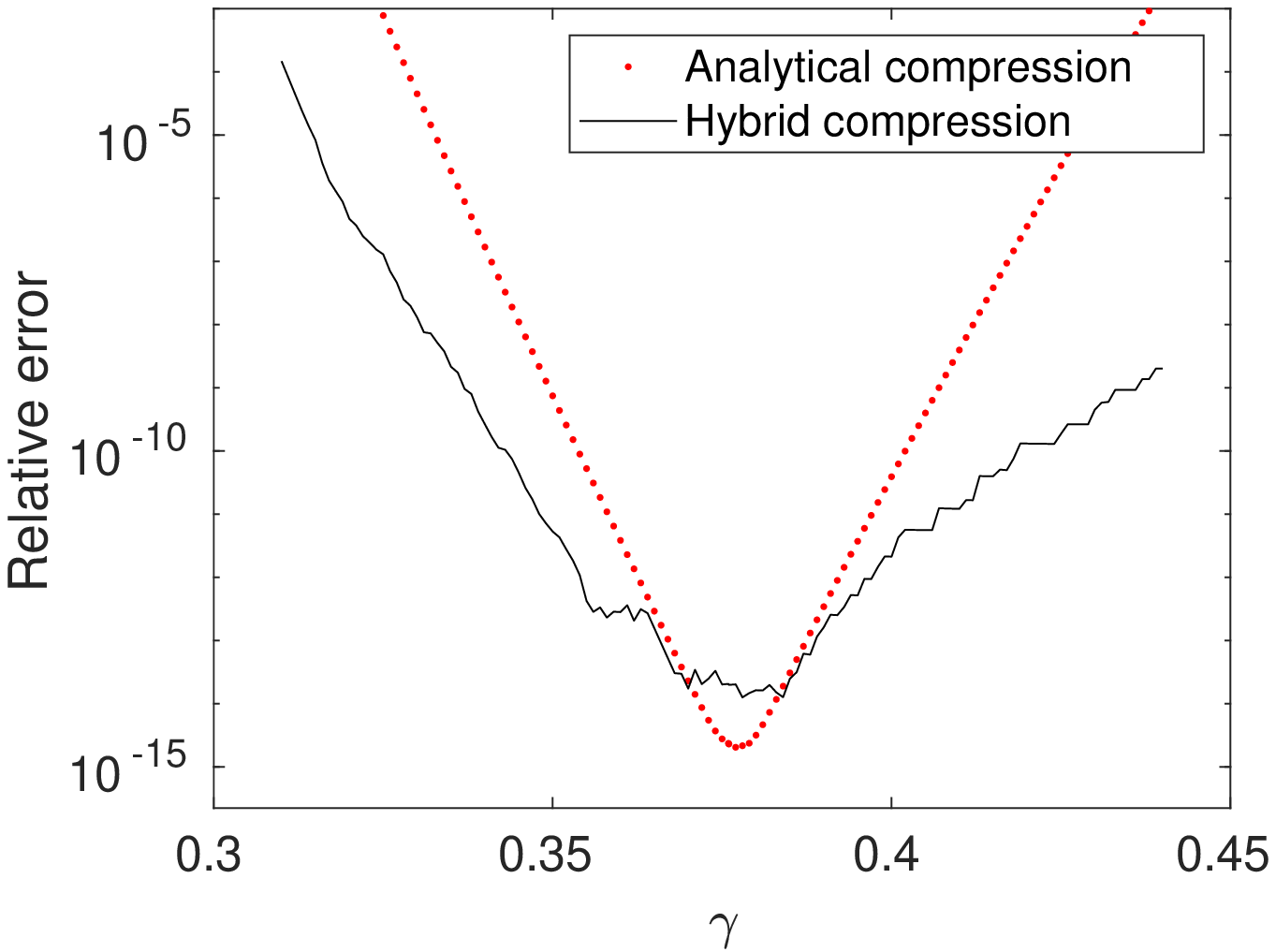}}
\subfloat[$d=4$]{\label{fig:test_1_4}\includegraphics[scale=.43]{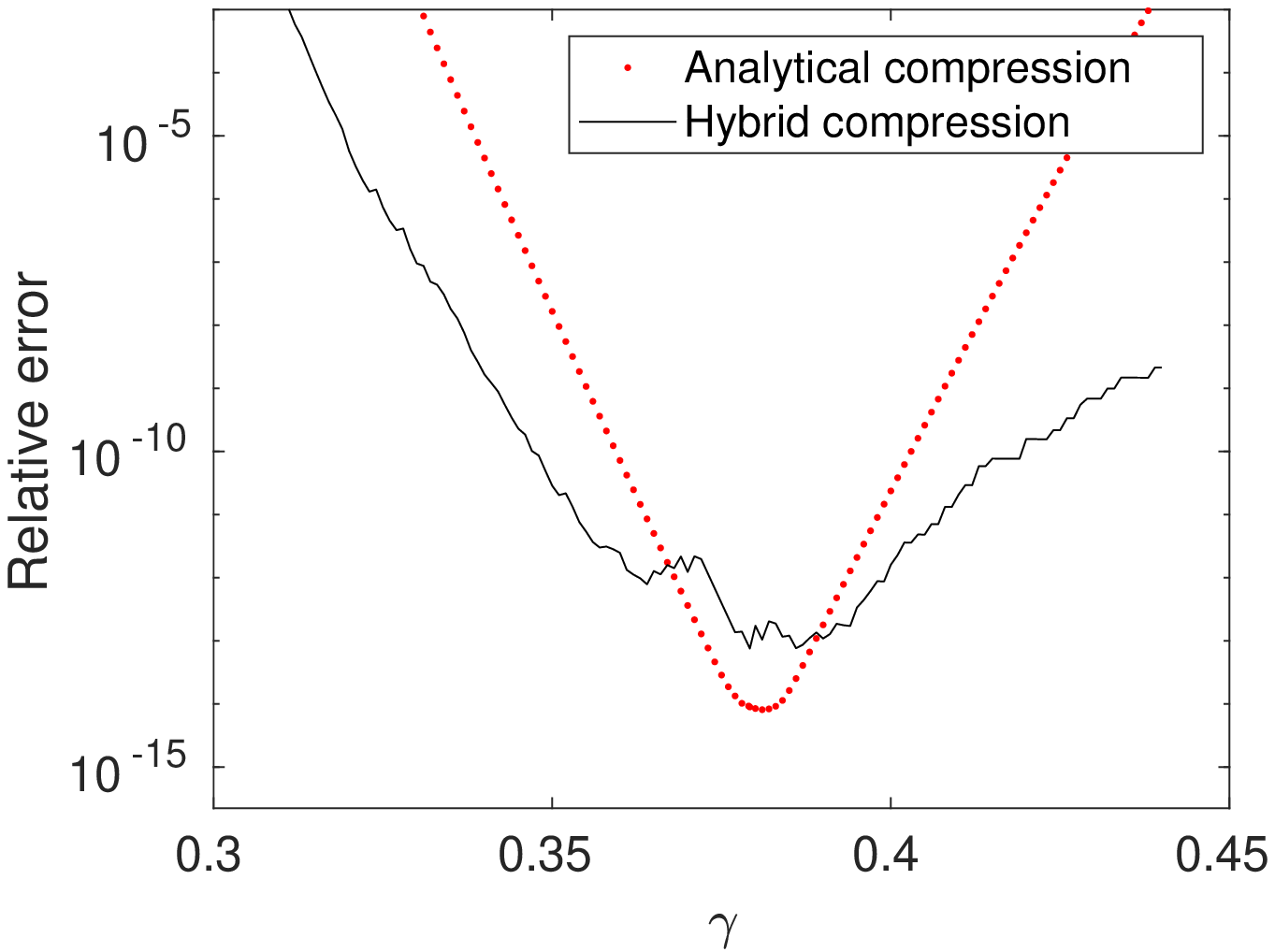}}
\caption{\cref{ex3}: $\mathcal{E}_N (\gamma)$ in the analytical compression step and $\mathcal{R}_N (\gamma)$ in the hybrid low-rank approximation with varying radius $\gamma$.}
\label{fig:ex3_1}
\end{figure}

Similar results are obtained for $d = 2, 3,4$. See \cref{fig:ex3_1} and \cref{tab:ex3_1}.
(We notice that $\mathcal{E}_N (\gamma)$ is sometimes larger than $\mathcal{R}_N (\gamma)$, especially
when $\gamma$ is closer to $X$ or $Y$. This
is likely due to the different amount of evaluations of the kernel function in the error computations.
The kernel function evaluations may have higher numerical errors when $\gamma$ gets closer to $\gamma_1$ or $\gamma_2$.
When $\gamma$ is not too close to $\gamma_1$ or $\gamma_2$, $\mathcal{R}_N (\gamma)$ is smaller than $\mathcal{E}_N (\gamma)$,
which is consistent with the theoretical estimates.
Here, no stabilization is integrated into the proxy point method (which may be fixed based on a technique in \cite{fmm1d}),
while SRRQR factorizations have full stability measurements and produce column basis matrices with controlled norms.
On the other hand, this also reflects that hybrid compression is a practical method.)

\begin{table}[tbhp]
    \centering
    \caption{\cref{ex3}: Hybrid compression results, where $\tilde{\gamma}^*$ is the approximate optimal radius.}
    \label{tab:ex3_1}
    \begin{tabular}{c|cccccc}\hline
    	$d$ & $N$ & Optimal $\gamma$ & $\tilde{\gamma}^*$ &  Numerical rank & $\mathcal{E}_N (\tilde{\gamma}^*)$ &$\mathcal{R}_N (\tilde{\gamma}^*)$  \\ \hline
    	$1$ & $169$ & $0.3678$ & $0.3675$ & $78$ &$3.2106E-16$ & $1.1008E-15$ \\
    	$2$ & $179$ & $0.3733$ & $0.3713$ & $88$ &$1.0431E-15$ & $2.1817E-15$ \\
    	$3$ & $187$ & $0.3774$ & $0.3759$ & $93$ &$2.3565E-15$ & $2.0537E-14$  \\
    	$4$ & $193$ & $0.3816$ & $0.3792$ & $99$ &$8.9381E-15$ & $7.5528E-14$  \\ \hline
    \end{tabular}
\end{table}

Also in \cref{fig:ex3_2} for $d=1,2$, we plot the proxy points as well as the representative points $\hat{X}$ produced by the
hybrid approximation with $\gamma = \tilde{\gamma}^*$.

\begin{figure}[tbhp]
\centering
\subfloat[$d=1$]{\label{fig:test_2_1}\includegraphics[height=1.8in]{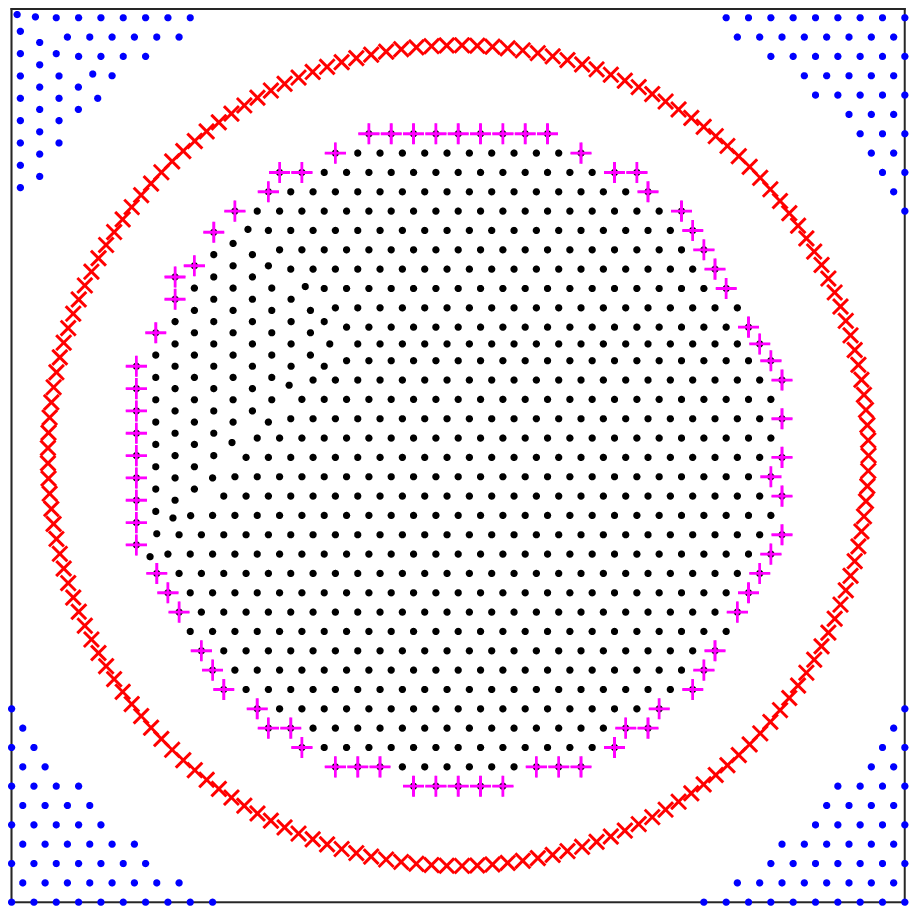}}
\subfloat[$d=2$]{\label{fig:test_2_2}\includegraphics[height=1.8in]{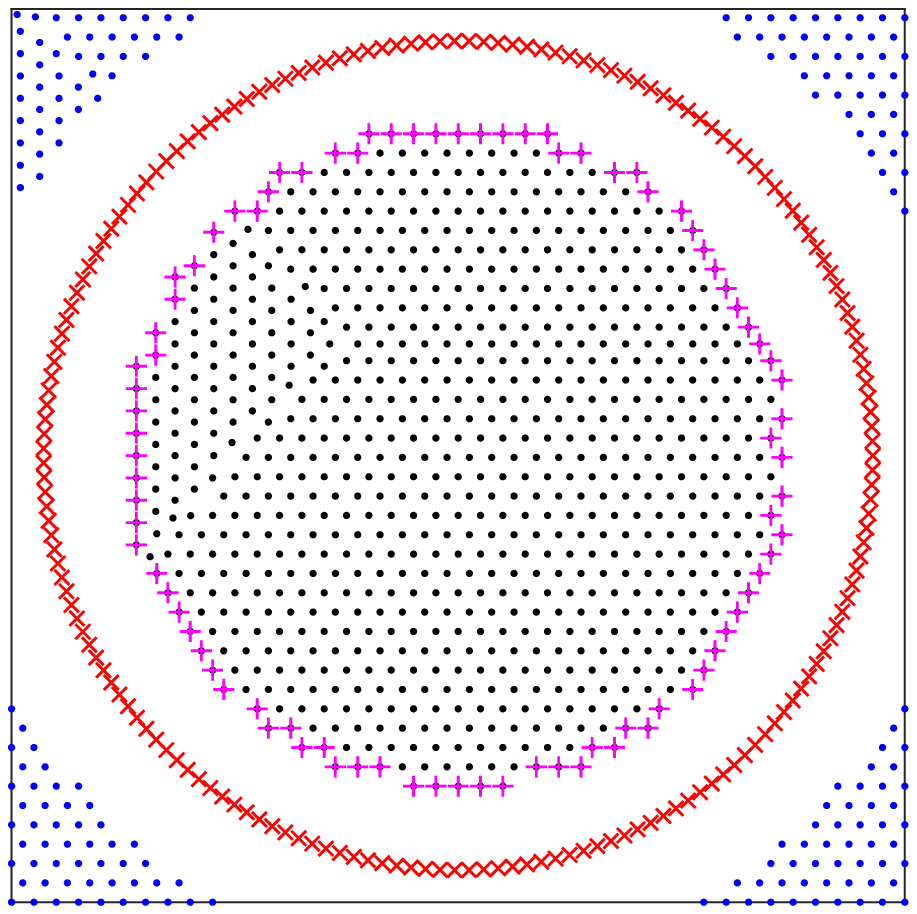}}
%\subfloat[$d=3$]{\label{fig:test_2_3}\includegraphics[height=1.6in]{rep_3}}
%\subfloat[$d=4$]{\label{fig:test_2_4}\includegraphics[height=1.6in]{rep_4}}
\caption{\cref{ex3}: Representative points ($+$ shapes) and proxy points ($\times$ shapes).}
\label{fig:ex3_2}
\end{figure}

In our next set of tests, we vary the number of proxy points $N$ for the analytical compression step and check its effect on the hybrid low-rank approximation error. For each $N$, the radius of the proxy surface $\gamma$ is set to be $\tilde{\gamma}^*$. The results are shown in \cref{fig:ex3_3}. The approximation error for the analytical compression decays exponentially as predicted by
\cref{thm:rel_err_bound_d1,cor:block_err} (until $N$ reaches the values indicated in \cref{tab:ex3_1}; after that point,
it stops to decay due to floating point errors).

\begin{figure}[tbhp]
\centering
\subfloat[$d=1$]{\label{fig:test_3_1}\includegraphics[scale=.43]{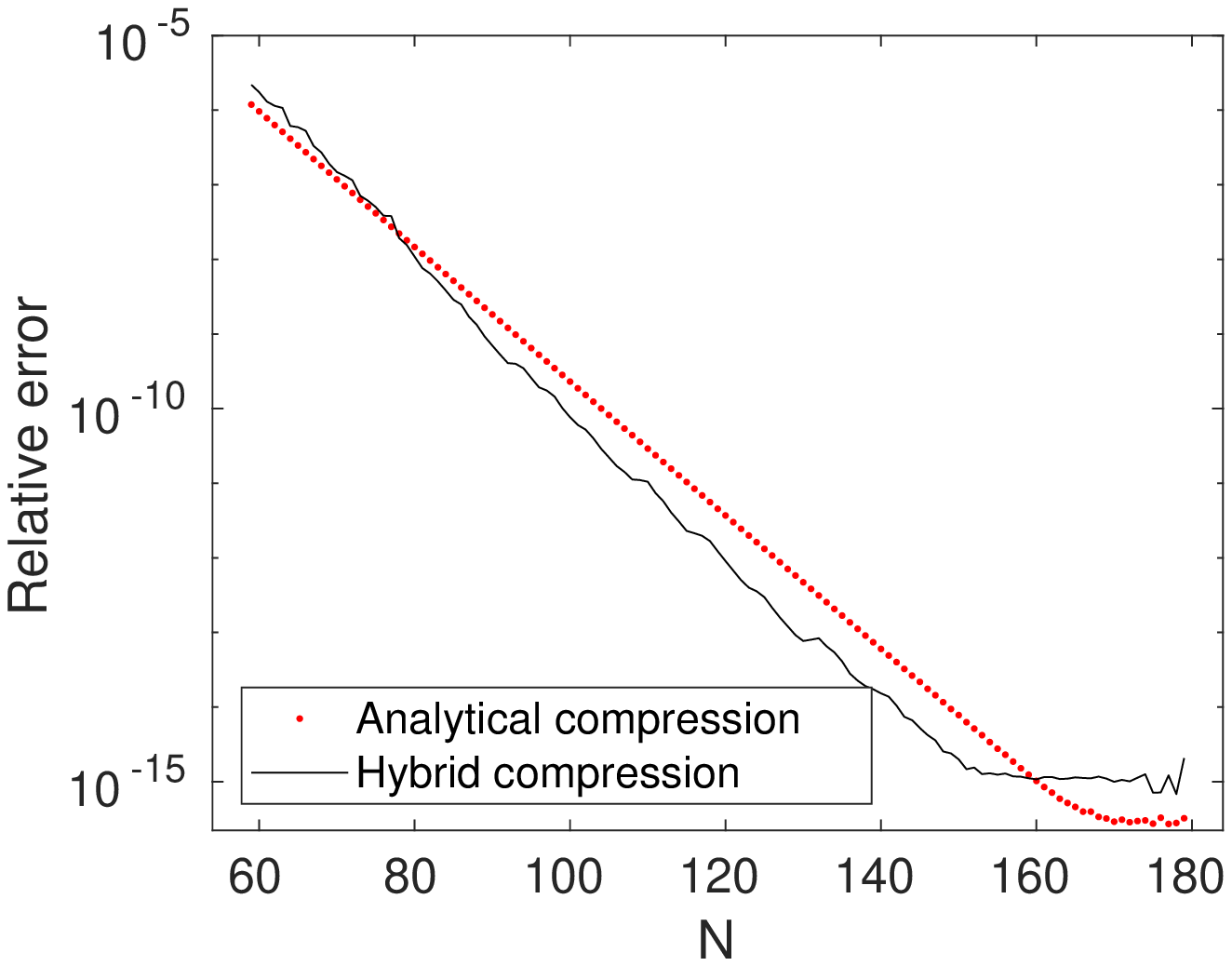}}
\subfloat[$d=2$]{\label{fig:test_3_2}\includegraphics[scale=.43]{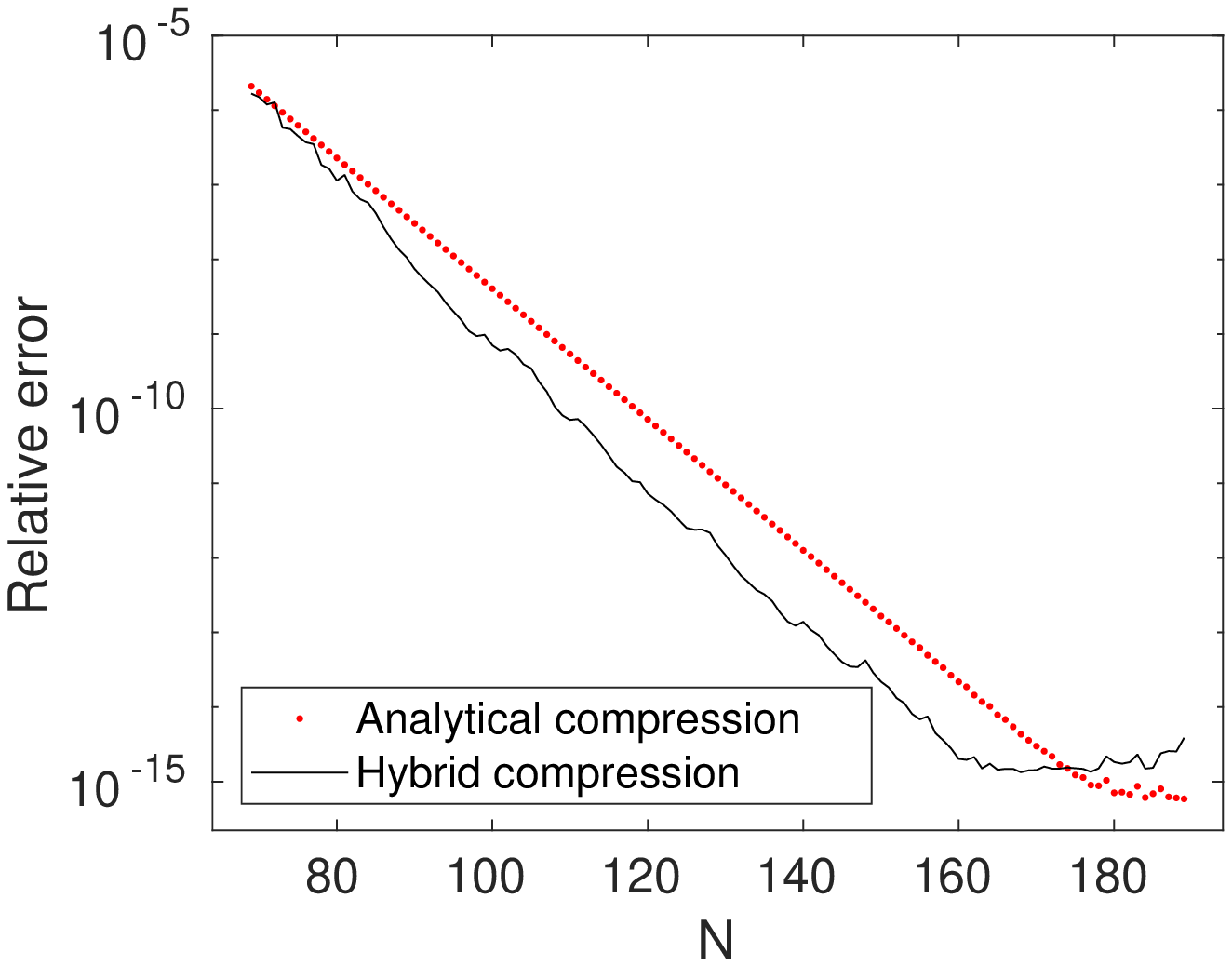}} \\
\subfloat[$d=3$]{\label{fig:test_3_3}\includegraphics[scale=.43]{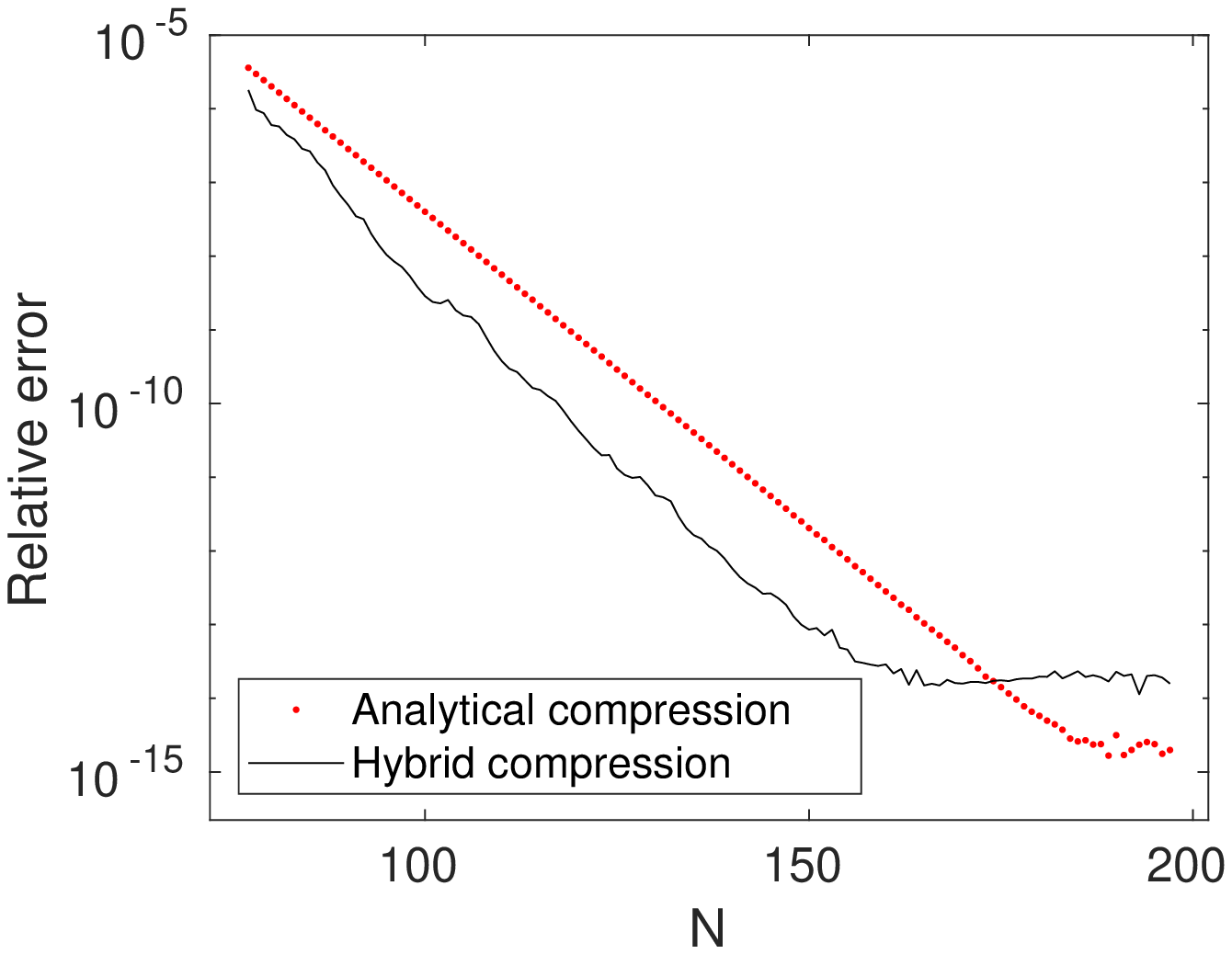}}
\subfloat[$d=4$]{\label{fig:test_3_4}\includegraphics[scale=.43]{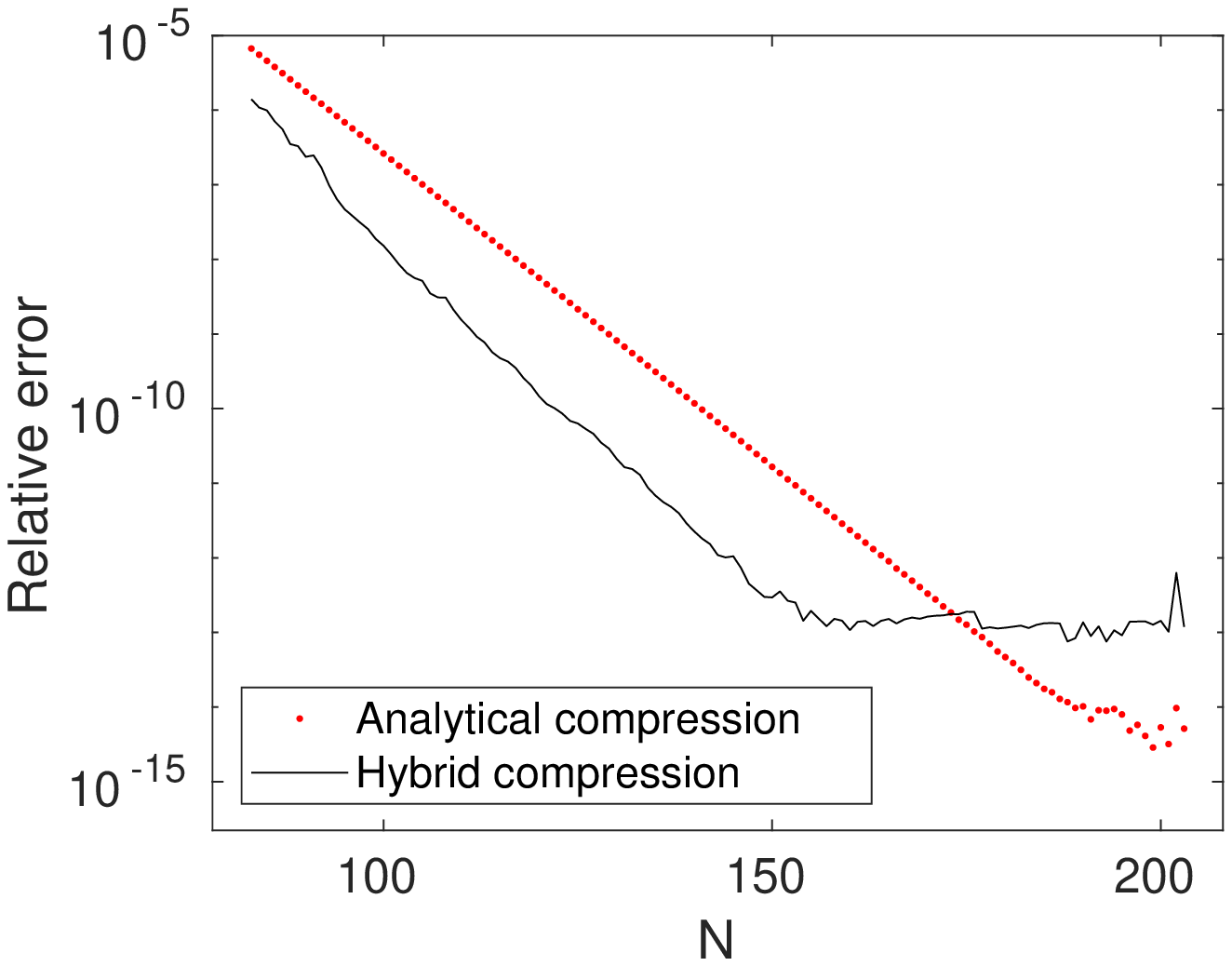}}
\caption{\cref{ex3}: Accuracies with $\gamma=\tilde{\gamma}^*$ and varying $N$.}
\label{fig:ex3_3}
\end{figure}

\section{Conclusions}
The proxy point method is a very simple and convenient strategy for computing low-rank approximations for kernel matrices
evaluated at well-separated sets. In this paper, we present an intuitive way of explaining
the method. Moreover, we provide rigorous approximation error analysis for the kernel function approximation
and low-rank kernel matrix approximation in terms of a class of important kernels. Based on the analysis, we show
how to choose nearly optimal locations of the proxy points. The work can serve as a starting point
to study the proxy point method for more general kernels and higher dimensions.
Some possible strategies in future work will be based on other kernel expansions or Cauchy FMM ideas \cite{letourneau2014cauchy}.
Various results here are already applicable to more general kernels and other approximation methods.
We also hope this work can draw more attentions from researchers in the field of matrix computations to study
and utilize such an elegant method.
% \appendix
% \appendixnotitle

\section*{Acknowledgments}
The authors would like to thank Steven Bell at Purdue University for some helpful discussions.

\bibliographystyle{siamplain}

\end{document}